\documentclass[11pt,english,british,refpage,intoc,bibliography=totoc,index=totoc,BCOR=7.5mm,captions=tableheading]{extarticle}
\usepackage{mathptmx}
\usepackage[T1]{fontenc}
\usepackage[latin9]{inputenc}
\usepackage[a4paper]{geometry}
\usepackage{color}
\usepackage{babel}
\usepackage{mathrsfs}
\usepackage{amsmath}
\usepackage{amsthm}
\usepackage{amssymb}
\usepackage{bm}
\usepackage[numbers]{natbib}
\usepackage[unicode=true,pdfusetitle,
 bookmarks=true,bookmarksnumbered=true,bookmarksopen=false,
 breaklinks=false,pdfborder={0 0 0},pdfborderstyle={},backref=false,colorlinks=true]
 {hyperref}
\hypersetup{
 linkcolor=black, citecolor=blue, urlcolor=black, filecolor=blue, pdfpagelayout=OneColumn, pdfnewwindow=true, pdfstartview=XYZ, plainpages=false}

\makeatletter
\numberwithin{equation}{section}
\numberwithin{figure}{section}
\theoremstyle{plain}
\newtheorem{thm}{\protect\theoremname}[section]
  \theoremstyle{definition}
  \newtheorem{defn}[thm]{\protect\definitionname}
  \theoremstyle{plain}
  \newtheorem{prop}[thm]{\protect\propositionname}
  \theoremstyle{plain}
  \newtheorem{lem}[thm]{\protect\lemmaname}

\@ifundefined{date}{}{\date{}}
\usepackage{caption}
\usepackage[nottoc]{tocbibind}
\allowdisplaybreaks[4]
\usepackage{dsfont}
\DeclareMathAlphabet{\mathcal}{OMS}{cmsy}{m}{n}

\makeatother

  \addto\captionsbritish{\renewcommand{\definitionname}{Definition}}
  \addto\captionsbritish{\renewcommand{\lemmaname}{Lemma}}
  \addto\captionsbritish{\renewcommand{\propositionname}{Proposition}}
  \addto\captionsbritish{\renewcommand{\theoremname}{Theorem}}
  \addto\captionsenglish{\renewcommand{\definitionname}{Definition}}
  \addto\captionsenglish{\renewcommand{\lemmaname}{Lemma}}
  \addto\captionsenglish{\renewcommand{\propositionname}{Proposition}}
  \addto\captionsenglish{\renewcommand{\theoremname}{Theorem}}
  \providecommand{\definitionname}{Definition}
  \providecommand{\lemmaname}{Lemma}
  \providecommand{\propositionname}{Proposition}
\providecommand{\theoremname}{Theorem}

\begin{document}
\global\long\def\Sgm{\boldsymbol{\Sigma}}

\global\long\def\W{\boldsymbol{W}}

\global\long\def\H{\mathscr{H}}

\global\long\def\P{\mathbb{P}}

\global\long\def\Q{\mathbb{Q}}

\title{Large deviation principle for fractional \\Brownian motion with
respect to capacity }

\author{Jiawei Li\thanks{Mathematical Institute, University of Oxford, Oxford OX2 6GG. Email:
\protect\href{mailto:jiawei.li@maths.ox.ac.uk}{jiawei.li@maths.ox.ac.uk}} \ and Zhongmin Qian\thanks{Mathematical Institute, University of Oxford, Oxford OX2 6GG. Email:
\protect\href{mailto:zhongmin.qian@maths.ox.ac.uk}{zhongmin.qian@maths.ox.ac.uk}}}
\maketitle
\begin{abstract}
We show that the fractional Brownian motion (fBM) defined via the Volterra
integral representation with Hurst parameter $H\geq\frac{1}{2}$ is
a quasi-surely defined Wiener functional on the classical Wiener space,
and we establish the large deviation principle (LDP) for such an fBM with
respect to $(p,r)$-capacity on the classical Wiener space in Malliavin's
sense.
\selectlanguage{english}%
\medskip

\emph{key words}: capacity, fractional Brownian motion, large deviations,
Malliavin derivative, quasi-sure analysis

\medskip

\end{abstract}

\section{Introduction}

Quasi-sure analysis, as a powerful tool to study functions on infinite dimensional spaces, was initiated by Malliavin  \cite{Malliavin1978a,Malliavin1978,Malliavin1997} using the stochastic calculus of variations and Fukushima \cite{Fukushima2010} by means of Dirichlet forms. The theory can be applied to plentiful aspects in stochastic analysis, such as Markov processes, Gaussian processes, large deviation principles  and etc., see e.g. \cite{Airault1991a,Boedihardjo2016,Denis1993,Fang1993,Fukushima1985,Gao2006, Gao2001,Huang1990,Kaneko1984,Liu2007,Ren1990,Ren1993,Yoshida1993} 
and the literature there-in. Malliavin observed that by constructing a regularity theory and a uniform measure on an abstract Wiener space, many interesting Wiener functionals are smooth, and this regularity theory enables us to study Wiener functionals like in the finite dimensional real analysis. An outer measure called $(p,r)$-capacity, denoted by $c_{p,r}$ throughout this paper, was introduced in terms of the Malliavin derivative and a number of papers concerning capacities have been published throughout last decades, see e.g. \cite{Fukushima1985,Kaneko1984,Kazumi1992a,Kazumi1992b,Malliavin1997,Malliavin1993a,Malliavin1993b}
and the literature there-in.

Among results related to quasi-sure analysis, the majority considers solutions to It{\^o}'s stochastic differential equations, which are merely measurable Wiener functionals. In fact, quasi-sure analysis can be used to handle processes which are  measurable functionals on the Wiener space but not solutions to any It{\^o}'s SDEs, such as fractional Brownian motions. Fractional Brownian motion (fBM), as an important example of Gaussian processes, has a variety of applications in mathematical finance, hydrodynamics, communication networks and so on, see e.g. \cite{Biagini2008, Decreusefond1999, Kolmogorov1940b, Mandelbrot1968,
Mishura2008}. To be more precise, an fBM $\left(B_t\right)_{t\geq 0}$ with Hurst parameter $H\in (0,1)$ is a centred self-similar Gaussian process, whose covariance function is given by 
\[\text{Cov}(s,t)=\frac{1}{2}\left(s^{2H}+t^{2H}-\left\vert t-s\right\vert^{2H}\right).
\]
By definition, one obtains immediately that when $H=\frac{1}{2}$, this process is a standard Brownian motion. However, when $H$ takes other values, this process differs from Brownian motion as its increments are no longer independent and thus exhibits long memory behaviour. Thanks to the Volterra integral representation introduced by Decreusefond and {\"U}st{\"u}nel \cite{Decreusefond1999}, which is 
\[
B_t=\int ^t _0 K(t,s)dW_s,\quad \forall t\geq 0,
\]
where $\left(W_t\right)_{t\geq 0}$ is a standard Brownian motion and $K$ is some singular kernel, fBMs can be regarded as measurable Wiener functionals. According to Malliavin, fBMs with different Hurst parameters induce a family of capacities living on distinct abstract Wiener spaces. Nonetheless, all these fBMs can be viewed as Wiener functionals on the classical Wiener space due to the integral representation, so that we can study them with one uniform measure - the capacity associated with Brownian motion. In this paper, we will prove that the integral representation of fBM is defined except for a capacity zero set when Hurst parameter $H\geq\frac{1}{2}$, that is, fBMs are quasi-surely defined Wiener functionals on the classical Wiener space.

In order to achieve this goal, we need several results from the rough paths theory. The analysis of rough paths, originated by Lyons (see e.g. \cite{Friz2014Course,Friz2010book,Lyons2002, Lyons2007Saint}), was established to study solutions to stochastic differential equations driven by semi-martingales and other rough signals. It turns out that many techniques developed in the rough paths analysis can be applied in the research of quasi-sure analysis as demonstrated in \cite{Boedihardjo2016}, as well as a series of work by various researchers (see e.g. \cite{inahama2006quasi,inahama2013laplace,inahama2015b,inahama2015a,Ledoux2002,millet2006large}
and etc.).

Besides proving that fBMs as Wiener functionals are quasi-surely defined, we establish large deviation principles for these Wiener functionals. 
Large deviations theory has been a prevalent topic in probability for its significance in statistics and statistical mechanics. It completes the central limit theorem by telling us that tail probabilities decay exponentially fast. In 1970s, the theory of large deviation principles (LDP for short) experienced rapid development due to the remarkable work by Donsker and Varadhan \cite{Donsker1974}, and one may refer to \cite{Dembo2009,Deuschel2001,Donsker1974,Varadhan1984} for further details. In finite dimensional case, one crucial result in this theory is the Cram{\'e}r's theorem, which precisely describes the rate of exponential decay. In infinite dimensional case, the exponential decay of large perturbations of Brownian motion from its mean trajectory is characterised in the Schilder's theorem, and the Freidlin-Wentzell theorem generalises it to the laws of It\^{o} diffusions, see e.g. \cite{Dembo2009} and \cite{Deuschel2001}. Similar results were proved using rough paths theory by Ledoux, Qian and Zhang in \cite{Ledoux2002}, see also \cite{friz2007large,millet2006large}. Moreover, the general Cram{\'e}r's theorem (see e.g. \cite{Deuschel2001}) can be used to study the large deviations of Gaussian measures. Indeed, large deviation principles can be formulated for not only measures, but also capacities. In \cite{Yoshida1993}, Yoshida established a version of LDP with respect to capacities on the abstract Wiener space, which implies the LDP for Gaussian measures, while in \cite{Gao2006} and \cite{Gao2001}, Gao and Ren considered the capacity version of Freidlin-Wentzell theorem. This line of research was taken a step further to the setting of Gaussian rough paths in \cite{Boedihardjo2016}. Inspired by the arguments in \cite{Boedihardjo2016}, we prove the LDP for fBMs with respect to $(p,r)$-capacity on the classical Wiener space.

The major difference between this paper and \cite{Yoshida1993} is that we use different capacities. Although the invariance property of capacities has been proved in \cite{Albeverio1992}, the capacities associated with different Gaussian measures are however non-comparable. Therefore, instead of using the capacities induced by fBMs, we treat fBMs with different Hurst parameters as a family of Wiener functionals and choose the Brownian motion capacity as a uniform measure.

This paper is organised in the following way. In next section, we present a few definitions and properties of capacities, and then we state the main result, a quasi-sure version of large deviation principle for fractional Brownian motions realised as Wiener functionals on the classical Wiener space. In section 3, we recall several elementary results such as Wiener chaos decomposition, exponential tightness, contraction principle in the context of quasi-sure analysis. Then in section 4, with a quite technical proof, we provide a construction of quasi-surely defined modifications of fBMs, which are considered to be Wiener functionals in our settings. Next, in section 5, we prove the exponential tightness of the family of finite dimensional approximations of fBMs (modified as in section 4). In this section, we adopt several ideas from the rough paths theory. Finally, we determine the rate function and complete our proof of the quasi-sure large deviation principles for fBMs in section 6. The key step is to obtain the finite dimensional quasi-sure large deviation principles, which may be accomplished by explicit computations.

\section{Preliminaries and the main result}

In this section, we will introduce basic definitions in Malliavin calculus and state the main result.

\subsection{Malliavin differentiation and capacities}

We mainly follow the notations used in Ikeda and Watanabe \cite{Ikeda2014}, and Nualart \cite{Nualart2006}. Although our presentation applies to multi-dimensional case as well, we only consider the one dimensional Wiener space here for simplicity. Let $\boldsymbol{W}$ be the space of all real-valued continuous paths over time interval $[0,1]$
starting from the origin, equipped with the uniform norm $\lVert \cdot \rVert$ given by $\lVert\omega\rVert=\sup_{t\in[0,1]}\vert\omega(t)\vert$, $\forall \omega\in \boldsymbol{W}$. The Borel $\sigma$-algebra is denoted by $\mathscr{B}(\boldsymbol{W})$. We call the functions which send each $\omega\in\boldsymbol{W}$ to its coordinates $\omega(t)$ (where $t\geq0$) the coordinate mapping processes, and are denoted by $\omega(t)$ or $\omega_{t}$. The Wiener measure $\P$ is the distribution of
standard Brownian motion, the unique probability on $\left(\boldsymbol{W},\mathscr{B}(\boldsymbol{W})\right)$
such that $\left\{ \omega(t):t\geq0\right\} $ is a standard Brownian
motion. Let $\mathscr{F}$ be the completion of $\mathscr{B}(\boldsymbol{W})$
under $\P$. The Wiener functionals, by convention in the literature, are
$\mathscr{F}$-measurable functions on $\W$. 

Let $\mathscr{H}$ denote the Cameron-Martin space, which is a
Hilbert space containing all absolutely continuous functions $h$ on $[0,1]$
such that $h(0)=0$ and its generalised derivative $\dot{h}$ is square
integrable. The inner product on $\mathscr{H}$ is given by
\[
\langle h,g\rangle_{\mathscr{H}}=\int_{0}^{1}\dot{h}(s)\dot{g}(s)ds,\quad \forall h,g\in \mathscr{H}.\]
This space may be embedded into $\boldsymbol{W}$ via a continuous
and dense embedding. Denote the topological dual of $\W$ by $\W^{*}$.
Then we have continuous dense embeddings $\W^{*}\hookrightarrow\mathscr{H}\hookrightarrow\W$.
The triple $(\W,\mathscr{H},\P)$ is called the classical Wiener
space.

There is a natural linear isometry from $\mathscr{H}$ to $L^{2}(\W)$ sending each element $h\in \mathscr{H}$ to a random variable $[h]$ such that $[h](\omega)=\int_{0}^{1}\dot{h}(s)d\omega(s)$ for all $\omega\in \W$, which is defined as an It\^o's
integral with respect to Brownian motion. Then $\mathbb{E}\left[[h][g]\right]=\langle h,g\rangle_{\mathscr{H}}$.
The family $\left\{ \left[h\right]:h\in\mathscr{H}\right\} $ is
called the Gaussian isometry process in \cite{Nualart2006} (see
Section 1.1.1, Chapter 1, \cite{Nualart2006}). 

A random variable $F$ on $\W$ is smooth if it is of the
form $F=f([h_{1}],\cdots,[h_{n}])$ for some $h_{1},\cdots h_{n}\in\mathscr{H}$
and $f\in C_{p}^{\infty}(\mathbb{R}^{n})$, a smooth function such
that $f$ and all of its partial derivatives have polynomial growth.
The collection of all smooth random variables is denoted by $\mathcal{S}$.
The Malliavin derivative $DF$ of $F$ is defined to be an $\mathscr{H}$-valued
random variable, given by
\[
DF=\sum_{i=1}^{n}\partial_{i}f([h_{1}],\cdots,[h_{n}])h_{i},
\]
where $\partial_{i}f$ denotes the partial derivative of $f$ in the
$i$-th component. The higher order derivatives, $D^{l}F$, $l\geq1$,
are defined inductively. The Sobolev norm $\lVert \cdot \rVert_{\mathbb{D}_{r}^{p}}$
of a smooth random variable $F$ is defined as
\[
\lVert F\rVert_{\mathbb{D}_{r}^{p}}=\left(\mathbb{E}\left[|F|^{p}\right]+\sum_{l=1}^{r}\mathbb{E}\left[\left|\lVert D^{l}F\rVert_{\mathscr{H}^{\otimes l}}\right|^{p}\right]\right)^{1/p},
\]
where $r=0,1,2,\cdots$ and $1\leq p<\infty$. The completion of $\mathcal{S}$
with respect to this norm is denoted by $\mathbb{D}_{r}^{p}$. 

The concept of capacities on the classical Wiener space plays a central role
in what follows. For given $p\geq1$ and $r=0,1,2,\cdots$, the capacity
$c_{p,r}$ is a sub-additive set function on the classical Wiener space, which
can be defined in two steps. 

First for every open $O\subset\W$ (see e.g. \cite{Malliavin1997}),
set

\[
c_{p,r}\left(O\right)=\inf\left\{ \lVert\varphi\rVert_{\mathbb{D}_{r}^{p}}:\varphi\in\mathbb{D}_{r}^{p},\ \varphi\geq1\ \text{a.e. on }O,\ \varphi\geq0\ \text{a.e. on }\boldsymbol{W}\right\} .
\]
Next for an arbitrary subset $A$ of $\boldsymbol{W}$, 
\[
c_{p,r}\left(A\right)=\inf\left\{ c_{p,r}\left(O\right):A\subset O,\ O\ \text{is open}\right\} .
\]

A property $\pi=\pi(\omega)$ (whose description depends on $\omega\in\boldsymbol{W}$)
holds $(p,r)$-quasi-surely (or simply $(p,r)$-q.s.), if the set
on which $\pi$ is not satisfied has $(p,r)$-capacity zero. A property
$\pi$ is said to hold quasi-surely (q.s.) if it holds $(p,r)$-quasi-surely
for all $r=0,1,2,\cdots$ and $1<p<\infty$. A set is said to be slim if it has zero $(p,r)$-capacity for all $r=0,1,2,\cdots$ and $1<p<\infty$.

Let us recall a few elementary properties of capacities $c_{p,r}$
(see e.g. Section 1.2, Chapter IV, \cite{Malliavin1997}). By definition
every capacity is an outer measure, so that $c_{p,r}(A)\leq c_{p,r}(B)$
for $A\subset B\subset\W$, and $c_{p,r}$ is sub-additive in that
\[
c_{p,r}\left(\bigcup_{n=1}^{\infty}A_{n}\right)\leq\sum_{n=1}^{\infty}c_{p,r}(A_{n})
\]
for any $A_{n}\subset\boldsymbol{W}$, $n=1,2,\cdots.$

The family of capacities $c_{p,r}$ are comparable in the sense that
for $\forall A\subset\boldsymbol{W}$, $1<p<q<\infty$, and $r<s$, we have
$c_{p,r}(A)\leq c_{q,r}(A)$ and $c_{p,r}(A)\le c_{p,s}(A)$. In particular,
by definition, $\left(\P(A)\right)^{\frac{1}{p}}\leq c_{p,r}(A)$
for any $A\in\mathscr{F}$.

The first Borel-Cantelli lemma can be proved in the context of capacities, and it says that if a sequence of subsets $\{A_{n}\}_{n=1}^{\infty}$ of $\boldsymbol{W}$
satisfies $\sum_{n=1}^{\infty}c_{p,r}(A_{n})<\infty$, then $c_{p,r}(\limsup_{n\to\infty}A_{n})=0$. One may refer to \cite{Malliavin1997} for a proof.  

Another important tool used in this paper is the capacity version
of Chebyshev's inequality, which is
\[
c_{p,r}\left(\varphi>\lambda\right)\leq\frac{\lVert\varphi\rVert_{\mathbb{D}_{r}^{p}}}{\lambda},
\]
for every lower semi-continuous $\varphi\in\mathbb{D}_{r}^{p}$ and $\lambda>0$. 

\subsection{The main result}

We are now ready to introduce the definition of large deviation
principle (LDP) with respect to capacities as in \cite{Boedihardjo2016} and \cite{Yoshida1993}. Then we define fractional Brownian motions and its integral representation according to \cite{Decreusefond1999} (see also Chapter 5 in \cite{Nualart2006}) and state the main result at the end of this section.
\begin{defn}
\label{def:LDP}Let $r\in\mathbb{N}$ and $p>1$, and $\left\{ X^{\varepsilon}:\varepsilon>0\right\} $
be a family of $(p,r)$-quasi-surely defined mappings from $\boldsymbol{W}$
to a Polish space $(Y,d)$. We say that the family $\{X^{\varepsilon}:\varepsilon>0\}$
satisfies the $c_{p,r}$-large deviation principle ($c_{p,r}$-LDP) with a good rate function $I:Y\to[0,\infty]$
if 

(1) $I$ is lower semi-continuous and for every $\alpha>0$, the
level set 
\[
\Psi_{I}(\alpha)=\left\{ y\in Y:I(y)\leq\alpha\right\} 
\]
is compact in $Y$; and

(2) for every closed $F\subset Y$, 
\[
\limsup_{\varepsilon\to0}\varepsilon^{2}\log c_{p,r}\left\{ \omega\in\boldsymbol{W}:X^{\varepsilon}(\omega)\in F\right\} \leq-\frac{1}{p}\inf_{y\in F}I(y),
\]
for every open $G\subset Y$,
\[
\limsup_{\varepsilon\to0}\varepsilon^{2}\log c_{p,r}\left\{ \omega\in\boldsymbol{W}:X^{\varepsilon}(\omega)\in G\right\} \geq-\frac{1}{p}\inf_{y\in G}I(y).
\]
\end{defn}

Recall that a fractional Brownian motion (fBM) $\left(B_{t}\right)_{t\geq0}$ with Hurst parameter
$H\in\left(0,1\right)$ is a self-similar Gaussian process
with stationary increments whose covariance function is given by
\[
\text{Cov} (s,t)=\mathbb{E}\left[B_{s}B_{t}\right]=\frac{1}{2}\left(t^{2H}+s^{2H}-|t-s|^{2H}\right).
\]
FBMs can be realised as Wiener functionals in the way that
\begin{equation}
B_{t}(\omega)=\int_{0}^{t}K(t,s)d\omega(s),\label{eq:integral_rep}
\end{equation}
where $\left\{ \omega(s):s\geq0\right\} $ is the coordinate mapping process
on $\boldsymbol{W}$ (hence a Brownian motion). $B_{t}$ is defined
almost surely \textendash{} the integral on the right-hand side is
understood in the It\^{o}'s sense. $K$ is a singular kernel, and when $H>\frac{1}{2}$, it is given by 
\[
K(t,s)=c_{H}s^{\frac{1}{2}-H}\int_{s}^{t}(u-s)^{H-\frac{3}{2}}u^{H-\frac{1}{2}}du,\quad s<t,\label{eq:K_repn}
\]
where $c_{H}$ is some constant depending only on
$H$. It is straightforward by the Kolmogorov continuity theorem that the process defined by (\ref{eq:integral_rep}) has a modification which is $\alpha$-H{\"o}lder continuous with $\alpha<H$. We abuse our notation by using $\left(B_{t}\right)_{t\geq0}$ to denote such a modification, and in the sequel, when we mention fBM, we always refer to this modified version. 

To state the quasi-sure large deviation principles for fBMs, we need to identify
the corresponding rate functions, which must be the same rate functions as in the
context of probabilities. Notice that (\ref{eq:integral_rep}) defines a mapping $B:\W\to \W$ taking almost all Brownian motion paths to the paths of fBM with Hurst parameter $H$. Let $\mathbb{Q}=\P\circ B^{-1}$ be the
push-forward of the Wiener measure, which is the distribution of this version of fBM, a Gaussian measure on
$\left(\W,\mathscr{B}(\W)\right)$. Similar to the case of Brownian motion, there is a canonical way to associate this Gaussian measure and $\W$ with a separable Hilbert space $\hat{\mathscr{H}}$, which can be continuously and densely embedded into $\W$ (see e.g. \cite{Bogachev1998}). In \cite{Decreusefond1999}, the Cameron-Martin space $\hat{\mathscr{H}}$ corresponding to the fractional Brownian motion with Hurst parameter $H$ was identified, and is the space consisting of all elements of the form $h(t)=\int ^t _0K(t,s)\dot{h}(s)ds$, where $K$ is as in (\ref{eq:K_repn}), and $\dot{h}\in L^2([0,1])$. The inner product on $\hat{\mathscr{H}}$ is defined as 
\[
\langle h_1,h_2 \rangle_{\hat{\mathscr{H}}}=\int^1_0 \dot{h}_1(s)\dot{h}_2(s)ds, 
\]
where ${h}_i(t)=\int^t _0K(t,s)\dot{h}_i(s)ds$, $i=1,2$. Then
\[
\int_{\W}e^{il(\omega)}\mathbb{Q}(d\omega)=e^{-\frac{\lVert l\rVert_{\hat{\mathscr{H}}}^{2}}{2}},\quad\forall l\in\W^{*}.
\]
Let $\left\{ \Q_{\varepsilon}\right\} $ be the family of scaled measures,
the laws of $\{\varepsilon\omega\}$ under $\Q$. By definition
\begin{equation}
\Q_{\varepsilon}(A)=\Q\left\{ \omega\in\W:\varepsilon\omega\in A\right\} =\P\left\{ \omega\in\W:\varepsilon B(\omega)\in A\right\} \label{32}
\end{equation}
for each $A\in\mathscr{B}(\W)$. According to Theorem 3.4.12 on page
88 in \cite{Deuschel2001}, $\{\Q_{\varepsilon}\}$ satisfies the
LDP with the good rate function $I$ given by 
\[
I(\omega)=\begin{cases}
\frac{\lVert\omega\rVert_{\hat{\mathscr{H}}}^{2}}{2}, & \omega\in\hat{\mathscr{H}},\\
\infty, & \text{otherwise}.
\end{cases}
\]

Now we are in a position to state the main result of the present paper. Let $(B_t)_{t\geq 0}$ be the continuous modification of fBM given by (\ref{eq:integral_rep}) with Hurst parameter $H$.
\begin{thm}
\label{thm:main LDP}Let $r\in\mathbb{N}$, $p>1$ and $\frac{1}{2}\leq H<1$.
Then we have the following conclusions.
\begin{enumerate}
\item[(1)] There exists a modification of $B_{t}$ (for $t\geq0$) defined $(p,r)$-quasi-surely.
\item[(2)] Let $X_{t}^{\varepsilon}\left(\omega\right)=B_{t}\left(\varepsilon\omega\right)$
for all $\omega$ expect for a $c_{p,r}$-zero subset (for all $t\geq0$,
$\varepsilon>0$). Then $\left\{ X^{\varepsilon}:\varepsilon>0\right\} $
(which are scaled fBMs with Hurst parameter $H\in\left[\frac{1}{2},1\right)$)
satisfies the $c_{p,r}$-LDP with the good rate function 
\begin{equation}
I(\omega)=\begin{cases}
\frac{\lVert\omega\rVert_{\hat{\mathscr{H}}}^{2}}{2}, & \omega\in\hat{\mathscr{H}},\\
\infty, & \text{otherwise}.
\end{cases}\label{eq:33}
\end{equation}
\end{enumerate}
\end{thm}

The remaining of this paper is devoted to the proof of the above result,
Theorem \ref{thm:main LDP}. The first part of Theorem
\ref{thm:main LDP} follows from Theorem \ref{thm:Th4-2} in Section
4 directly. The proof of the second part will be presented in Section 5 and Section 6. 

Our strategy of the proof is the following. For each
fixed $t\in[0,1]$ and $m\in\mathbb{N}$, we consider $B_{t}^{(m)}$, a finite linear combination of elements in the classical Wiener space, and show that the sequence
$(B_{t}^{(m)})_{m\in\mathbb{N}}$ converges quasi-surely
(with respect to the Brownian motion capacity). The main difficulty
here is that the kernel $K(t,s)$ is singular in $s$, so it is very
difficult to control its increments in $s$ and estimate the integral
of $K$ over small time intervals near time $s=0$. However, we notice
that $K$ as a function of $t$, behaves more regularly. Therefore,
we control the difference $K(t,s+\alpha)-K(t,s)$ by the difference
of $K$ when $t$ varies. Then we may obtain the desired mapping $X$, a quasi-surely defined modification of fBM on the Wiener space by approximations with linear interpolations. Define
\[
X^{(n)}(\omega)(t):=B_{\frac{k-1}{2^{n}}}(\omega)+2^{n}\left(t-\frac{k-1}{2^{n}}\right)\left(B_{\frac{k}{2^{n}}}(\omega)-B_{\frac{k-1}{2^{n}}}(\omega)\right),\quad\forall\frac{k-1}{2^{n}}\leq t\leq\frac{k}{2^{n}}.
\]
For each $n$, $X^{(n)}$ is quasi-surely defined, then we show that
this sequence $(X^{(n)})_{n\in\mathbb{N}}$ converges to some mapping $X$ quasi-surely.
Since any countable union of capacity zero sets still has zero capacity,
we conclude that the limit $X$ is quasi-surely defined, and we shall
also see that this convergence is exponentially fast in the proof.
Now as the large deviation principle may be established for $X^{(n)}$'s,
using exponentially good approximations result from the LDP theory, we
deduce the $c_{p,r}$-LDP for the limit mapping $X$. 

\section{Some technical facts}

In this section we collect a few technical facts which will be used
to prove the quasi-sure version of large deviation principles for
fBMs.

\subsection{Wiener chaos decomposition}

The $n$-th Wiener chaos $\mathcal{H}_{n}$, $n=0,1,2,\cdots$,
is the closed subspace of $L^{2}(\W)$ generated by all random variables
$H_{n}([h])$, $\forall h\in\mathscr{H}$ with $\lVert h\rVert_{\mathscr{H}}=1$,
where $H_{n}(x)$ is the $n$-th Hermite polynomial. $\mathcal{H}_{n}$
and $\mathcal{H}_{m}$ are orthogonal subspaces of $L^{2}(\W)$ when
$m\neq n$. One important result is the Wiener chaos decomposition, which states that
\[
L^{2}(\W,\mathscr{G},\P)=\oplus_{n=0}^{\infty}\mathcal{H}_{n},
\]
where $\mathscr{G}$ is the $\sigma$-algebra generated by random variables $\{[h]:h\in \mathscr{H}\}$ (see e.g. Theorem 1.1.1, Section 1.1 in \cite{Nualart2006} and its
proof). The projection from $L^{2}(\W)$ to $\mathcal{H}_{n}$ is
denoted by $J_{n}$. Let $\mathcal{P}_{n}^{0}$ denote the space of
polynomial random variables of the form
\[
F=p([h_{1}],\cdots[h_{k}]),\quad\forall h_{1},\cdots h_{k}\in\mathscr{H},
\]
where $p$ is a polynomial with degree less than or equal to $n$,
and let $\mathcal{P}_{n}$ be the closure of $\mathcal{P}_{n}^{0}$ in $L^{2}(\W)$. Then it follows that for every integer $n$, $\mathcal{P}_{n}=\oplus_{m=0}^{n}\mathcal{H}_{m}$. If $F\in\mathbb{D}_{r}^{p}$ for some $p\geq2$, then for all $l\leq r$,
\begin{equation}
\left\Vert \lVert D^{l}F\rVert_{\mathscr{H}^{\otimes l}}\right\Vert _{2}^{2}=\sum_{n=l}^{\infty}n(n-1)\cdots(n-l+1)\lVert J_{n}F\rVert_{2}^{2}.\label{eq:deriv_decomp_L2norm}
\end{equation}
For a proof of this relation, one may refer to Section 1.2 in \cite{Nualart2006}.

To establish next proposition, we need to apply the hypercontractivity property of Ornstein-Uhlenbeck semigroup. The Ornstein-Uhlenbeck semigroup $(T_{t})_{t\geq0}$ is a semigroup of contractions on $L^{2}(\W)$ defined by 
\[
T_{t}F=\sum_{n=0}^{\infty}e^{-nt}J_{n}F.
\]
This definition can be extended to $L^p(\W)$ and for every $p>1$, $(T_{t})_{t\geq0}$ is defined to be a semigroup of contractions on $L^{p}(\W)$, where for each $t\geq0$,
\[
T_{t}F(\omega)=\int_{\W}F(e^{-t}\omega+\sqrt{1-e^{-2t}}x)\P(dx).
\]
This semigroup enjoys the hypercontractivity property, which is  $\lVert T_{t}F\rVert_{q(t)}\leq\lVert F\rVert_{p}$ for all $F\in L^{p}(\W)$, where $p>1$, $t>0$ and $q(t)=e^{2t}(p-1)+1>p$. For a proof, see e.g. Theorem 1.4.1, Section 1.4, \cite{Nualart2006}. 
\begin{prop}
\label{prop:Lp to L2}Let $F\in\mathcal{P}_{n}$. Then
\begin{equation}
\lVert F\rVert_{q}\leq(n+1)(q-1)^{\frac{n}{2}}\lVert F\rVert_{2}\label{eq:Lp_control_L2}
\end{equation}
and
\begin{equation}
\left\Vert \lVert D^{l}F\rVert_{\mathscr{H}^{\otimes l}}\right\Vert _{2}\leq n^{\frac{l}{2}}\lVert F\rVert_{2}\label{eq:deriv_norm_L2_norm}
\end{equation}
 for any $q>2$ and $l\le n$.
\end{prop}

\begin{proof}
The first inequality (\ref{eq:Lp_control_L2}) results from
hypercontractivity. Take $F\in\mathcal{H}_{n}$. Then $T_{t}F=e^{-nt}F$.
Set $p=2$, and $q=q(t)=1+e^{2t}$, so $t=\frac{1}{2}\log(q-1)$,
and hence by hypercontractivity, 
\[
\lVert T_{t}F\rVert_{q}=\lVert e^{-\frac{n}{2}\log(q-1)}F\rVert_{q}=(q-1)^{-\frac{n}{2}}\lVert F\rVert_{q}\leq\lVert F\rVert_{2},
\]
which implies that
\[
\lVert F\rVert_{q}\leq(q-1)^{\frac{n}{2}}\lVert F\rVert_{2}.
\]
Now let $F=\sum_{m=0}^{n}J_{m}F\in\mathcal{P}_{n}$, then
\[
\lVert F\rVert_{q}\leq\sum_{m=0}^{n}\lVert J_{m}F\rVert_{q}\leq\sum_{m=0}^{n}(q-1)^{\frac{m}{2}}\lVert J_{m}F\rVert_{2}\leq(n+1)(q-1)^{\frac{n}{2}}\lVert F\rVert_{2}.
\]

To prove (\ref{eq:deriv_norm_L2_norm}), we apply (\ref{eq:deriv_decomp_L2norm})
to $F\in\mathcal{P}_{n}$, where $F\in\mathbb{D}_{l}^{2}$ for $l\leq n$.
Notice that $J_{m}F$'s vanish when $m>n$, and thus $\lVert F\rVert_{2}^{2}=\sum_{m=0}^{n}\lVert J_{m}F\rVert_{2}^{2}$,
so that
\[
\left\Vert \lVert D^{l}F\rVert_{\mathscr{H}^{\otimes l}}\right\Vert _{2}^{2}=\sum_{m=l}^{n}m(m-1)\cdots(m-l+1)\lVert J_{m}F\rVert_{2}^{2}\leq n^{l}\sum_{m=l}^{n}\lVert J_{m}F\rVert_{2}^{2}\leq n^{l}\lVert F\rVert_{2}^{2}.
\]

\end{proof}

\subsection{Exponential tightness}

Most conclusions in the theory of large deviations (see \cite{Dembo2009,Deuschel2001}
for details) are still valid in the context of capacities. Let us
state some of them which will be used in this paper, their proofs
are routine and will be omitted. 
\begin{prop}
\label{thm:Contraction-Principle}\emph{(Varadhan's Contraction Principle)}
Let $\{X^{\varepsilon}:\varepsilon>0\}$ be a family of $(p,r)$-quasi-surely
defined maps from $\boldsymbol{W}$ to a Polish space $(Y_{1},d_{1})$ satisfying
the $c_{p,r}$-LDP with the good rate function $I$. Let $F$ be a continuous
map from $(Y_{1},d_{1})$ to another Polish space $(Y_{2},d_{2})$. Then
the family $\left\{ F\circ X^{\varepsilon}:\varepsilon>0\right\} $
of $(p,r)$-quasi-surely defined maps satisfies the $c_{p,r}$-LDP
with the good rate function 
\[
J(z)=\inf_{y:F(y)=z}I(y),
\]
where $\inf\emptyset=\infty$.
\end{prop}

To deal with fBMs, which are merely measurable Wiener functionals, the concept of exponential tightness is a useful technique in proving large deviation principles. The natural modification of this notion can be formulated as the following. 
\begin{defn}
For $m=1,2,\cdots$, let $\left\{ X^{\varepsilon,(m)}:\varepsilon>0\right\}$ 
and $\{X^{\varepsilon}:\varepsilon>0\}$ be families of $(p,r)$-quasi-surely defined
mappings from $\W$ to some Polish space $(Y,d)$. Then $\left\{ X^{\varepsilon,(m)}:\varepsilon>0\right\} $
is said to be a family of exponentially good approximations of $\{X^{\varepsilon}:\varepsilon>0\}$
under $(p,r)$-capacity if for all $\lambda>0$, 
\begin{equation}
\lim_{m\to\infty}\limsup_{\varepsilon\to0}\varepsilon^{2}\log c_{p,r}\left\{ \omega:d\left(X^{\varepsilon,(m)}(\omega),X^{\varepsilon}(\omega)\right)>\lambda\right\} =-\infty.\label{eq:exp-rate}
\end{equation}
\end{defn}

The following version of the contraction principle will be useful in our proof.
\begin{prop}
\label{prop:Exp-good-approx}Suppose that for each $m=1,2,\cdots$,
the family $\left\{ X^{\varepsilon,(m)}:\varepsilon>0\right\} $,
consisting of $(p,r)$-quasi-surely defined mappings from $\W$ to $\left(Y,d\right)$,
satisfies $c_{p,r}$-LDP with the good rate function $I_{m}$, and $\left\{ X^{\varepsilon,(m)}:\varepsilon>0\right\} $
are exponentially good approximations of $(p,r)$-quasi-surely defined mappings
$\left\{ X^{\varepsilon}:\varepsilon>0\right\} $. Define the function
\begin{equation}
J(y)=\sup_{\lambda>0}\liminf_{m\to\infty}\inf_{z\in B(y,\lambda)}J_{m}(z),\quad\forall y\in Y,\label{eq:34}
\end{equation}
where $B(y,\lambda)$ denotes the open ball in $(Y,d)$ with centre
$y$ and radius $\lambda$. If $J$ is a good rate function and for
every closed $C\subset Y$, 
\begin{equation}
\inf_{y\in C}J(y)\leq\limsup_{m\to\infty}\inf_{y\in C}J_{m}(y),\label{eq:35}
\end{equation}
then the family $\left\{ X^{\varepsilon}:\varepsilon>0\right\} $
satisfies $c_{p,r}$-LDP with the good rate function $J$.
\end{prop}

\section{FBMs as Wiener functionals}

Let $B=\left(B_{t}\right)_{t\geq0}$ the continuous version of fBM with Hurst parameter $H$ defined via (\ref{eq:integral_rep}). According to the transfer principle in \cite{Nualart2006} (see Proposition 5.2.1, Section 5.2, \cite{Nualart2006}), $B_{t}\in\mathbb{D}_{r}^{p}$,
and its Malliavin derivative may be computed explicitly as in the
following lemma (see also \cite{LI2019141}).
\begin{lem}
\label{lem:Malliavin deriv}Let $H\in(0,1)$, $r\in\mathbb{N}$ and
$p\in(1,\infty)$. Then for every $t>0$, $B_{t}\in\mathbb{D}_{r}^{p}$ and its first order Malliavin derivative is given by 
\[
DB_{t}(s)=\int_{0}^{s\wedge t}K(t,u)du.
\]
The higher order derivatives of $B_{t}$ vanish.
\end{lem}

For every $t>0$, and $m=1,2,\cdots$, define
\begin{equation}
B_{t}^{(m)}(\omega):=\sum_{i=0}^{2^{m}-1}\frac{2^{m}}{t}\int_{\frac{i}{2^{m}}t}^{\frac{i+1}{2^{m}}t}K(t,r)dr\left(\omega_{\frac{i+1}{2^{m}}t}-\omega_{\frac{i}{2^{m}}t}\right),\label{eq:11}
\end{equation}
where $B_{0}^{(m)}=0$. Obviously $B_{t}^{(m)}\in\mathbb{D}_{r}^{p}$
and 
\begin{equation}
DB_{t}^{(m)}(s)=u_{t}^{(m)}=\sum_{i=0}^{2^{m}-1}\frac{2^{m}}{t}\int_{\frac{i}{2^{m}}t}^{\frac{i+1}{2^{m}}t}K(t,r)dr\mathds{1}_{\left(\frac{i}{2^{m}}t,\frac{i+1}{2^{m}}t\right]}(s).\label{eq:10}
\end{equation}
Its higher order Malliavin derivatives vanish identically. 

The first part of Theorem \ref{thm:main LDP} is a
consequence of the following result.
\begin{thm}
\label{thm:Th4-2}Let $H\in\left[\frac{1}{2},1\right)$. For all $r\in\mathbb{N}$, $1<p<\infty$ and $t\in[0,1]$,
$(B_{t}^{(m)})_{m\in\mathbb{N}}$ converges $(p,r)$-quasi-surely
to some limit, denoted by $B_{t}$ too, which is also the limit of
$(B_{t}^{(m)})_{m\in\mathbb{N}}$ in $\mathbb{D}_{r}^{p}$.
\end{thm}

\begin{proof}
The proof is quite technical and will be divided into several steps. When $H=\frac{1}{2}$, an fBM is a standard Brownian motion, and hence the result follows immediately. We only need to consider the case when $H>\frac{1}{2}$. Let us begin with a simple fact that 
\[
\left\{ \omega:(B_{t}^{(m)}(\omega))_{m\in \mathbb{N}}\text{ is not Cauchy}\right\} \subset\limsup_{m\to\infty}\left\{ \omega:\left\vert B_{t}^{(m+1)}(\omega)-B_{t}^{(m)}(\omega)\right\vert >\frac{1}{2^{m\delta}}\right\} 
\]
for some $\delta>0$. Therefore, by the first Borel-Cantelli lemma for capacities, we only
need to show that 
\[
\sum_{m=1}^{\infty}c_{p,r}\left( \left|B_{t}^{(m+1)}-B_{t}^{(m)}\right|>\frac{1}{2^{m\delta}}\right) <\infty
\]
for all $p\in(1,\infty)$ and $r\in\mathbb{N}$. Since $c_{p,r}$
is increasing in $p$ and $r$, it suffices to prove that the above
infinite sum is finite for $p>2$ and all $r\in\mathbb{N}$. Therefore,
we shall assume that $p>2$ in the sequel. 

\textbf{Step 1.} In this step, we convert our problem from estimating the
capacities to estimating the $L^{2}$-norm of Gaussian random variables.
By Chebyshev's inequality, we have
\begin{align}
c_{p,r}\left(\left\vert B_{t}^{(m+1)}-B_{t}^{(m)}\right\vert >\lambda\right) & =c_{p,r}\left(\left\vert B_{t}^{(m+1)}-B_{t}^{(m)}\right\vert ^{2}>\lambda^{2}\right)\nonumber \\
 & \leq\lambda^{-2}\left\Vert \left(B_{t}^{(m+1)}-B_{t}^{(m)}\right)^{2}\right\Vert _{\mathbb{D}_{r}^{p}}\label{eq:6}
\end{align}
for any $\lambda>0$. Since $\left(B_{t}^{(m+1)}-B_{t}^{(m)}\right)^{2}$
is a polynomial functional of degree $2$, and 
\[
D\left(B_{t}^{(m+1)}-B_{t}^{(m)}\right)=u_{t}^{(m+1)}-u_{t}^{(m)},
\]
where $u_{t}^{(m)}$ is defined as in (\ref{eq:10}), so for all $l\geq3$, 
\[D^{l}\left(\left(B_{t}^{(m+1)}-B_{t}^{(m)}\right)^{2}\right)=0.\]
Therefore, by (\ref{prop:Lp to L2}),
\begin{align*}
\left\Vert \left\Vert D^{l}\left(\left(B_{t}^{(m+1)}-B_{t}^{(m)}\right)^{2}\right)\right\Vert _{\mathscr{H}^{\otimes l}}\right\Vert _{p} & \leq3(p-1)\left\Vert \left\Vert D^{l}\left(\left(B_{t}^{(m+1)}-B_{t}^{(m)}\right)^{2}\right)\right\Vert _{\mathscr{H}^{\otimes l}}\right\Vert _{2}\\
 & \leq3(p-1)2^{\frac{l}{2}}\left\Vert \left(B_{t}^{(m+1)}-B_{t}^{(m)}\right)^{2}\right\Vert _{2}\\
 & =3(p-1)2^{\frac{l}{2}}\left\Vert B_{t}^{(m+1)}-B_{t}^{(m)}\right\Vert _{4}^{2}\\
 & \leq3(p-1)2^{\frac{l}{2}}\left(2\sqrt{3}\left\Vert B_{t}^{(m+1)}-B_{t}^{(m)}\right\Vert _{2}\right)^{2}\\
 & =36(p-1)2^{\frac{l}{2}}\left\Vert B_{t}^{(m+1)}-B_{t}^{(m)}\right\Vert _{2}^{2}
\end{align*}
for all $p>2$ and $0\leq l\leq r$, and therefore
\begin{align}
\left\Vert \left(B_{t}^{(m+1)}-B_{t}^{(m)}\right)^{2}\right\Vert _{\mathbb{D}_{r}^{p}} & \leq\sum_{l=0}^{r}\left\Vert \left\Vert D^{l}\left(B_{t}^{(m+1)}-B_{t}^{(m)}\right)\right\Vert _{\mathscr{H}^{\otimes l}}\right\Vert _{p}\nonumber \\
 & \leq C_{r,p}\left\Vert B_{t}^{(m+1)}-B_{t}^{(m)}\right\Vert _{2}^{2},\label{eq:7}
\end{align}
where 
\[C_{r,p}=36(r+1)(p-1)2^{\frac{r}{2}}\]
depends only on $r$ and $p$. The $L^{2}$-norm on the right-hand side of (\ref{eq:7}) can be handled as the following. By definition (\ref{eq:11}),
\begin{align*}
B_{t}^{(m+1)}(\omega)-B_{t}^{(m)}(\omega) & =\sum_{i=0}^{2^{m+1}-1}\frac{2^{m+1}}{t}\int_{\frac{i}{2^{m+1}}t}^{\frac{i+1}{2^{m+1}}t}K(t,r)dr\left(\omega_{\frac{i+1}{2^{m+1}}t}-\omega_{\frac{i}{2^{m+1}}t}\right)\\
 & \hphantom{=}-\sum_{i=0}^{2^{m}-1}\frac{2^{m}}{t}\int_{\frac{i}{2^{m}}t}^{\frac{i+1}{2^{m}}t}K(t,r)dr\left(\omega_{\frac{i+1}{2^{m}}t}-\omega_{\frac{i}{2^{m}}t}\right).
\end{align*}
The integral term in $B_{t}^{(m)}(\omega)$ may be split into two
parts, i.e. for each $i\in\{0,1,\cdots,2^{m}-1\}$, 
\begin{align*}
 & \hphantom{=\ \ }\frac{2^{m}}{t}\int_{\frac{i}{2^{m}}t}^{\frac{i+1}{2^{m}}t}K(t,r)dr\left(\omega_{\frac{i+1}{2^{m}}t}-\omega_{\frac{i}{2^{m}}t}\right)\\
 & =\frac{2^{m}}{t}\left[\int_{\frac{2i+1}{2^{m+1}}t}^{\frac{2i+2}{2^{m+1}}t}K(t,r)dr\left(\omega_{\frac{2i+2}{2^{m+1}}t}-\omega_{\frac{2i}{2^{m+1}}t}\right)+\int_{\frac{2i}{2^{m+1}}t}^{\frac{2i+1}{2^{m+1}}t}K(t,r)dr\left(\omega_{\frac{2i+2}{2^{m+1}}t}-\omega_{\frac{2i}{2^{m+1}}t}\right)\right],
\end{align*}
 %& =\frac{2^{m}}{t}\left(\int_{\frac{2i+1}{2^{m+1}}t}^{\frac{2i+2}{2^{m+1}}t}K(t,r)dr+\int_{\frac{2i}{2^{m+1}}t}^{\frac{2i+1}{2^{m+1}}t}K(t,r)dr\right)\left(\omega_{\frac{i+1}{2^{m}}t}-\omega_{\frac{i}{2^{m}}t}\right)\\
and the contribution from the interval $\left(\frac{i}{2^{m}}t,\frac{i+1}{2^{m}}t\right]$
in $B_{t}^{(m+1)}(\omega)$ is 
\[
\frac{2^{m+1}}{t}\left[\int_{\frac{2i+1}{2^{m+1}}t}^{\frac{2i+2}{2^{m+1}}t}K(t,r)dr\left(\omega_{\frac{2i+2}{2^{m+1}}t}-\omega_{\frac{2i+1}{2^{m+1}}t}\right)+\int_{\frac{2i}{2^{m+1}}t}^{\frac{2i+1}{2^{m+1}}t}K(t,r)dr\left(\omega_{\frac{2i+1}{2^{m+1}}t}-\omega_{\frac{2i}{2^{m+1}}t}\right)\right].
\]
Therefore, we deduce that 
\begin{align}
 & \hphantom{=\ \ }B_{t}^{(m+1)}(\omega)-B_{t}^{(m)}(\omega)\nonumber \\
 & =\frac{2^{m}}{t}\sum_{i=0}^{2^{m}-1}\left[\int_{\frac{2i+1}{2^{m+1}}t}^{\frac{2i+2}{2^{m+1}}t}K(t,r)dr\left(\omega_{\frac{2i+2}{2^{m+1}}t}-2\omega_{\frac{2i+1}{2^{m+1}}t}+\omega_{\frac{2i}{2^{m+1}}t}\right)\right.\nonumber \\
 & \hphantom{=}\left.+\int_{\frac{2i}{2^{m+1}}t}^{\frac{2i+1}{2^{m+1}}t}K(t,r)dr\left(2\omega_{\frac{2i+1}{2^{m+1}}t}-\omega_{\frac{2i}{2^{m+1}}t}-\omega_{\frac{2i+2}{2^{m+1}}t}\right)\right]\nonumber \\
 & =\frac{2^{m}}{t}\sum_{i=0}^{2^{m}-1}\left(\int_{\frac{2i+1}{2^{m+1}}t}^{\frac{2i+2}{2^{m+1}}t}K(t,r)dr-\int_{\frac{2i}{2^{m+1}}t}^{\frac{2i+1}{2^{m+1}}t}K(t,r)dr\right)\left(\omega_{\frac{2i+2}{2^{m+1}}t}-2\omega_{\frac{2i+1}{2^{m+1}}t}+\omega_{\frac{2i}{2^{m+1}}t}\right).\label{eq:1}
\end{align}
Since Brownian motion has independent increments, we have 
\begin{equation}
\left\Vert B_{t}^{(m+1)}-B_{t}^{(m)}\right\Vert _{2}^{2}\nonumber =\left(\frac{2^{m}}{t}\right)\sum_{i=0}^{2^{m}-1}\left(\int_{\frac{2i+1}{2^{m+1}}t}^{\frac{2i+2}{2^{m+1}}t}K(t,r)dr-\int_{\frac{2i}{2^{m+1}}t}^{\frac{2i+1}{2^{m+1}}t}K(t,r)dr\right)^{2}.\label{eq:5}
\end{equation}
 %& =\mathbb{E}\left[\left(\frac{2^{m}}{t}\sum_{i=0}^{2^{m}-1}\left(\int_{\frac{2i+1}{2^{m+1}}t}^{\frac{2i+2}{2^{m+1}}t}K(t,r)dr-\int_{\frac{2i}{2^{m+1}}t}^{\frac{2i+1}{2^{m+1}}t}K(t,r)dr\right)\left(\omega_{\frac{2i+2}{2^{m+1}}t}-2\omega_{\frac{2i+1}{2^{m+1}}t}+\omega_{\frac{2i}{2^{m+1}}t}\right)\right)^{2}\right]\nonumber \\
\textbf{Step 2.} In this step, we further simplify our problem using
a rather simple observation. By change of variables, for each $i\in\left\{ 0,\cdots,2^{m}-1\right\} $,
\begin{align*}
M_{i} & :=\int_{\frac{2i+1}{2^{m+1}}t}^{\frac{2i+2}{2^{m+1}}t}K(t,r)dr-\int_{\frac{2i}{2^{m+1}}t}^{\frac{2i+1}{2^{m+1}}t}K(t,r)dr\\
 & =\int_{\frac{2i}{2^{m+1}}t}^{\frac{2i+1}{2^{m+1}}t}K\left(t,s+\frac{t}{2^{m+1}}\right)-K(t,s)ds.
\end{align*}
% & =\int_{\frac{2i}{2^{m+1}}t}^{\frac{2i+1}{2^{m+1}}t}K\left(t,s+\frac{t}{2^{m+1}}\right)ds-\int_{\frac{2i}{2^{m+1}}t}^{\frac{2i+1}{2^{m+1}}t}K(t,r)dr\\
Using the definition of $K$ and change of variables, we observe that for all $\alpha\in(0,t-s)$,

\begin{align*}
K(t,s+\alpha) & =c_{H}(s+\alpha)^{\frac{1}{2}-H}\int_{s+\alpha}^{t}(u-s-\alpha)^{H-\frac{3}{2}}u^{H-\frac{1}{2}}du\\
 & =c_{H}(s+\alpha)^{\frac{1}{2}-H}\int_{s}^{t-\alpha}(v-s)^{H-\frac{3}{2}}v^{H-\frac{1}{2}}\left(\frac{v+\alpha}{v}\right)^{H-\frac{1}{2}}dv\\
 & \leq c_{H}s^{\frac{1}{2}-H}\int_{s}^{t-\alpha}(v-s)^{H-\frac{3}{2}}v^{H-\frac{1}{2}}dv\\
\vphantom{\int_{a}^{b}} & =K(t-\alpha,s),
\end{align*}
and hence 
\begin{equation}
K(t,s+\alpha)-K(t,s)\leq K(t-\alpha,s)-K(t,s).\label{eq:diff_upp_bd}
\end{equation}
On the other hand, for every $\alpha\in(0,r)$, 
\begin{align*}
K(t,r-\alpha) & =c_{H}(r-\alpha)^{\frac{1}{2}-H}\int_{r-\alpha}^{t}(u-r+\alpha)^{H-\frac{3}{2}}u^{H-\frac{1}{2}}du\\
 & \leq c_{H}(r-\alpha)^{\frac{1}{2}-H}\int_{r}^{t+\alpha}(v-r)^{H-\frac{3}{2}}v^{H-\frac{1}{2}}dv\\
 & =\left(\frac{r}{r-\alpha}\right)^{H-\frac{1}{2}}K(t+\alpha,r).
\end{align*}
By setting $r=s+\alpha$, we deduce that
\begin{align}
K(t,s+\alpha)-K(t,s) & =K(t,s+\alpha)-K(t,s+\alpha-\alpha)\nonumber \\
 & \geq K(t,s+\alpha)-\left(\frac{s+\alpha}{s}\right)^{H-\frac{1}{2}}K(t+\alpha,s+\alpha).\label{eq:diff_lower_bd}
\end{align}
Now let $\alpha=\frac{t}{2^{m+1}}$ in (\ref{eq:diff_upp_bd}) and
(\ref{eq:diff_lower_bd}), and for all $s<t-\frac{t}{2^{m+1}}$, define
\begin{align*}
L_{i} & =\int_{\frac{2i}{2^{m+1}}t}^{\frac{2i+1}{2^{m+1}}t}K\left(t,s+\frac{t}{2^{m+1}}\right)-\left(\frac{s+\frac{t}{2^{m+1}}}{s}\right)^{H-\frac{1}{2}}K\left(t+\frac{t}{2^{m+1}},s+\frac{t}{2^{m+1}}\right)ds\\
 & =\int_{\frac{2i+1}{2^{m+1}}t}^{\frac{2i+2}{2^{m+1}}t}K(t,r)-\left(\frac{r}{r-\frac{t}{2^{m+1}}}\right)^{H-\frac{1}{2}}K\left(t+\frac{t}{2^{m+1}},r\right)dr,
\end{align*}
and 
\begin{align*}
U_{i}:= & \int_{\frac{2i}{2^{m+1}}t}^{\frac{2i+1}{2^{m+1}}t}K\left(t-\frac{t}{2^{m+1}},s\right)-K(t,s)ds.
\end{align*}
Then $L_{i}\leq M_{i}\leq U_{i}$ for each $i$, and it thus follows
that $M_{i}^{2}\leq L_{i}^{2}\vee U_{i}^{2}$, which implies that
\begin{align}
\left\Vert B_{t}^{(m+1)}-B_{t}^{(m)}\right\Vert _{2}^{2} & =\left(\frac{2^{m}}{t}\right)\sum_{i=0}^{2^{m}-1}M_{i}^{2}\nonumber \\
 & \leq\left(\left(\frac{2^{m}}{t}\right)\sum_{i=0}^{2^{m}-1}L_{i}^{2}\right)\vee\left(\left(\frac{2^{m}}{t}\right)\sum_{i=0}^{2^{m}-1}U_{i}^{2}\right).\label{eq:L2_bd}
\end{align}

\textbf{Step 3.} In this step, we find upper bounds for $L_{i}^{2}$
and $U_{i}^{2}$, respectively. We first find a control of
\[
\left(\frac{2^{m}}{t}\right)\sum_{i=0}^{2^{m}-1}L_{i}^{2}=\left(\frac{2^{m}}{t}\right)\sum_{i=0}^{2^{m}-1}\left(\int_{\frac{2i+1}{2^{m+1}}t}^{\frac{2i+2}{2^{m+1}}t}K(t,r)-\left(\frac{r}{r-\frac{t}{2^{m+1}}}\right)^{H-\frac{1}{2}}K\left(t+\frac{t}{2^{m+1}},r\right)dr\right)^{2}.
\]
For all $r\in\left(\frac{t}{2^{m+1}},t\right)$, consider the function

\begin{align*}
f_{r}(x)= & \left(\frac{r}{r-x}\right)^{H-\frac{1}{2}}K(t+x,r),\quad0\leq x<r.
\end{align*}
Then $f_{r}(0)=K(t,r)$ and 
\[
f_{r}\left(\frac{t}{2^{m+1}}\right)=r^{H-\frac{1}{2}}\left(r-\frac{t}{2^{m+1}}\right)^{\frac{1}{2}-H}K\left(t+\frac{t}{2^{m+1}},r\right).
\]
Therefore,
\[
L_{i}=\int_{\frac{2i+1}{2^{m+1}}t}^{\frac{2i+2}{2^{m+1}}t}\int_{\frac{t}{2^{m+1}}}^{0}f_{r}'(x)dxdr.
\]
We may compute the derivative of $f_{r}$, which is
\begin{align*}
f_{r}'(x) & =r^{H-\frac{1}{2}}\left[\left(H-\frac{1}{2}\right)(r-x)^{-\frac{1}{2}-H}K(t+x,r)+(r-x)^{\frac{1}{2}-H}\partial_{1}K(t+x,r)\right]\\
 & =\left(H-\frac{1}{2}\right)r^{H-\frac{1}{2}}(r-x)^{-\frac{1}{2}-H}K(t+x,r)+c_{H}(r-x)^{\frac{1}{2}-H}\left(t+x\right)^{H-\frac{1}{2}}(t+x-r)^{H-\frac{3}{2}},
\end{align*}
where $\partial_{1}K(t,s)$ denotes the partial derivative of $K$
with respect to the first variable. Denote 
\[
g_{r}(x)=\left(H-\frac{1}{2}\right)r^{H-\frac{1}{2}}(r-x)^{-\frac{1}{2}-H}K(t+x,r)
\]
and 
\[
h_{r}(x)=c_{H}(r-x)^{\frac{1}{2}-H}\left(t+x\right)^{H-\frac{1}{2}}(t+x-r)^{H-\frac{3}{2}}.
\]
Then for all $x<r$, $g_{r}(x)\geq0$ and $h_{r}(x)\geq0$. By H\"{o}lder's
inequality, we obtain that
\begin{align*}
L_{i}^{2} 
 & \leq\frac{t}{2^{m+1}}\int_{\frac{2i+1}{2^{m+1}}t}^{\frac{2i+2}{2^{m+1}}t}\left(\int_{0}^{\frac{t}{2^{m+1}}}g_{r}(x)dx+\int_{0}^{\frac{t}{2^{m+1}}}h_{r}(x)dx\right)^{2}dr\\
 & \leq\frac{t}{2^{m}}\int_{\frac{2i+1}{2^{m+1}}t}^{\frac{2i+2}{2^{m+1}}t}\left(\int_{0}^{\frac{t}{2^{m+1}}}g_{r}(x)dx\right)^{2}+\left(\int_{0}^{\frac{t}{2^{m+1}}}h_{r}(x)dx\right)^{2}dr,
\end{align*}
%& =\left(-\int_{\frac{2i+1}{2^{m+1}}t}^{\frac{2i+2}{2^{m+1}}t}\int_{0}^{\frac{t}{2^{m+1}}}f_{r}'(x)dxdr\right)^{2}\\
% & =\left(\int_{\frac{2i+1}{2^{m+1}}t}^{\frac{2i+2}{2^{m+1}}t}\left(\int_{0}^{\frac{t}{2^{m+1}}}g_{r}(x)dx+\int_{0}^{\frac{t}{2^{m+1}}}h_{r}(x)dx\right)dr\right)^{2}\\
and hence 
\begin{align}
\left(\frac{2^{m}}{t}\right)\sum_{i=0}^{2^{m}-1}L_{i}^{2}\leq & \int_{\frac{t}{2^{m+1}}}^{t}\left(\int_{0}^{\frac{t}{2^{m+1}}}g_{r}(x)dx\right)^{2}+\left(\int_{0}^{\frac{t}{2^{m+1}}}h_{r}(x)dx\right)^{2}dr.\label{eq:g+h_bound}
\end{align}
We control the integral of $g_{r}$ first. When $H>\frac{1}{2}$,
since $t\leq1$, 
we have $
K(t,s)s^{H-\frac{1}{2}}  \leq c_{H},
$ 
i.e. $K(t,s)\leq c_{H}s^{\frac{1}{2}-H}$. Therefore, we deduce that
\begin{align*}
0\leq\int_{0}^{\frac{t}{2^{m+1}}}g_{r}(x)dx & =\left(H-\frac{1}{2}\right)\int_{0}^{\frac{t}{2^{m+1}}}r^{H-\frac{1}{2}}(r-x)^{-\frac{1}{2}-H}K(t+x,r)dx\\
 & \leq c_{H}\left(H-\frac{1}{2}\right)\int_{0}^{\frac{t}{2^{m+1}}}(r-x)^{-\frac{1}{2}-H}dx\\
 & =c_{H}\left(\left(r-\frac{t}{2^{m+1}}\right)^{\frac{1}{2}-H}-r^{\frac{1}{2}-H}\right).
\end{align*}
As $\left(r-\frac{t}{2^{m+1}}\right)^{\frac{1}{2}-H}\geq r^{\frac{1}{2}-H}$,
it follows that
\begin{align*}
\left(\left(r-\frac{t}{2^{m+1}}\right)^{\frac{1}{2}-H}-r^{\frac{1}{2}-H}\right)^{2}  \leq\left(r-\frac{t}{2^{m+1}}\right)^{1-2H}-r^{1-2H}.
\end{align*}
%& =\left(r-\frac{t}{2^{m+1}}\right)^{1-2H}-2\left(r-\frac{t}{2^{m+1}}\right)^{\frac{1}{2}-H}r^{\frac{1}{2}-H}+r^{1-2H}\\
% &
Consequently,
\begin{align}
\left(\int_{0}^{\frac{t}{2^{m+1}}}g_{r}(x)dx\right)^{2} & \leq C_{1}\left[\left(r-\frac{t}{2^{m+1}}\right)^{1-2H}-r^{1-2H}\right],\label{eq:g_bound}
\end{align}
where $C_{1}$ is a constant depending only on $H$. As for $h_{r}$,
due to change of variables, 
\begin{align*}
0\leq\int_{0}^{\frac{t}{2^{m+1}}}h_{r}(x)dx & =c_{H}\int_{0}^{\frac{t}{2^{m+1}}}(r-x)^{\frac{1}{2}-H}\left(t+x\right)^{H-\frac{1}{2}}(t+x-r)^{H-\frac{3}{2}}dx\vphantom{\left(\frac{t}{2^{m+1}}\right)^{H-\frac{1}{2}}}\\
\vphantom{\left(\frac{t}{2^{m+1}}\right)^{H-\frac{1}{2}}} & \leq c_{H}\left(2t\right)^{H-\frac{1}{2}}\int_{0}^{\frac{t}{2^{m+1}}}(r-x)^{\frac{1}{2}-H}(t+x-r)^{H-\frac{3}{2}}dx\\
\vphantom{\left(\frac{t}{2^{m+1}}\right)^{H-\frac{1}{2}}} & =c_{H}\left(2t\right)^{H-\frac{1}{2}}\int_{r-\frac{t}{2^{m+1}}}^{r}y^{\frac{1}{2}-H}(t-y)^{H-\frac{3}{2}}dy\\
 & \leq\frac{2c_{H}}{H-\frac{1}{2}}\left(r-\frac{t}{2^{m+1}}\right)^{\frac{1}{2}-H}\left(\left(t+\frac{t}{2^{m+1}}-r\right)^{H-\frac{1}{2}}-\left(t-r\right)^{H-\frac{1}{2}}\right)\\
 & \leq\frac{2c_{H}}{H-\frac{1}{2}}\left(r-\frac{t}{2^{m+1}}\right)^{\frac{1}{2}-H}\left(\frac{t}{2^{m+1}}\right)^{H-\frac{1}{2}},
\end{align*}
% & \leq c_{H}\left(2t\right)^{H-\frac{1}{2}}\left(r-\frac{t}{2^{m+1}}\right)^{\frac{1}{2}-H}\int_{r-\frac{t}{2^{m+1}}}^{r}(t-y)^{H-\frac{3}{2}}dy\\
which implies that
\begin{align}
\left(\int_{0}^{\frac{t}{2^{m+1}}}h_{r}(x)dx\right)^{2} & \leq C_{2}\left(\frac{t}{2^{m+1}}\right)^{2H-1}\left(r-\frac{t}{2^{m+1}}\right)^{1-2H},\label{eq:h_bound}
\end{align}
with $C_{2}$ some constant only depending on the value of $H$. Using
(\ref{eq:g+h_bound}), we get that
\begin{align}
\left(\frac{2^{m}}{t}\right)\sum_{i=0}^{2^{m}-1}L_{i}^{2} & \leq C_{1}\int_{\frac{t}{2^{m+1}}}^{t}\left(r-\frac{t}{2^{m+1}}\right)^{1-2H}-r^{1-2H}dr\nonumber \\
 & \hphantom{=}+C_{2}\left(\frac{t}{2^{m+1}}\right)^{2H-1}\int_{\frac{t}{2^{m+1}}}^{t}\left(r-\frac{t}{2^{m+1}}\right)^{1-2H}dr\nonumber \\
 & =C_{1}\frac{1}{2-2H}\left(\left(t-\frac{t}{2^{m+1}}\right)^{2-2H}-t^{2-2H}+\left(\frac{t}{2^{m+1}}\right)^{2-2H}\right)\nonumber \\
 & \hphantom{=}+C_{2}\left(\frac{t}{2^{m+1}}\right)^{2H-1}\frac{1}{2-2H}\left(t-\frac{t}{2^{m+1}}\right)^{2-2H}\nonumber \\
 & \leq c_{1}\left(\frac{t}{2^{m+1}}\right)^{2-2H}+c_{2}\left(\frac{t}{2^{m+1}}\right)^{2H-1},\label{eq:lower_bd}
\end{align}
where $c_{1}$ and $c_{2}$ are two positive constants. 

Next, we move onto the estimate for $U_{i}$'s. By the definition
of $U_{i}$'s and H\"{o}lder's inequality, we have that

\begin{align}
  \hphantom{=\ \ }\left(\frac{2^{m}}{t}\right)\sum_{i=0}^{2^{m}-1}U_{i}^{2}\nonumber 
 & \leq\left(\frac{2^{m}}{t}\right)\sum_{i=0}^{2^{m}-1}\left(\frac{t}{2^{m+1}}\right)\int_{\frac{2i}{2^{m+1}}t}^{\frac{2i+1}{2^{m+1}}t}\left(K(t-\frac{t}{2^{m+1}},s)-K(t,s)\right)^{2}ds\nonumber \\
 & \leq\frac{1}{2}\int_{0}^{t-\frac{t}{2^{m+1}}}K^{2}(t-\frac{t}{2^{m+1}},s)ds-\int_{0}^{t-\frac{t}{2^{m+1}}}K(t-\frac{t}{2^{m+1}},s)K(t,s)ds\nonumber \\
\vphantom{\left(\frac{t}{2^{m+1}}\right)} & \hphantom{=}+\frac{1}{2}\int_{0}^{t}K^{2}(t,s)ds\nonumber \\
 & =\frac{1}{2}\left(\frac{t}{2^{m+1}}\right)^{2H}.\label{eq:upper_bd}
\end{align}

\textbf{Step 4.} In this step, we complete our proof using the above estimates. It follows from (\ref{eq:L2_bd}), (\ref{eq:lower_bd}) and
(\ref{eq:upper_bd}) that
\begin{align*}
\left\Vert B_{t}^{(m+1)}-B_{t}^{(m)}\right\Vert _{2}^{2} & \leq\left(c_{1}\left(\frac{t}{2^{m+1}}\right)^{2-2H}+c_{2}\left(\frac{t}{2^{m+1}}\right)^{2H-1}\right)\vee\frac{1}{2}\left(\frac{t}{2^{m+1}}\right)^{2H}\\
 & \leq c_{3}\left(\frac{t}{2^{m+1}}\right)^{2-2H}+c_{4}\left(\frac{t}{2^{m+1}}\right)^{2H-1},
\end{align*}
 where $c_{3}$ and $c_{4}$ are some constants. 

Therefore, for any $\lambda>0$, it holds that 
\begin{align*}
c_{p,r}\left(\left\vert B_{t}^{(m+1)}-B_{t}^{(m)}\right\vert >\lambda\right) & \leq C_{r,p}\lambda^{-2}\left\Vert B_{t}^{(m+1)}-B_{t}^{(m)}\right\Vert _{2}^{2}\vphantom{\left(\frac{t}{2^{m+1}}\right)^{2H-1}}\\
 & \leq C_{r,p}\lambda^{-2}\left(c_{3}\left(\frac{t}{2^{m+1}}\right)^{2-2H}+c_{4}\left(\frac{t}{2^{m+1}}\right)^{2H-1}\right).
\end{align*}
Set $\lambda=2^{-m\delta}$, then as $t\leq1$,
\begin{align*}
c_{p,r}\left(\left\vert B_{t}^{(m+1)}-B_{t}^{(m)}\right\vert >\frac{1}{2^{m\delta}}\right) & \leq C_{r,p,H}\left(\frac{1}{2^{2m(1-H-\delta)}}+\frac{1}{2^{2m(H-\frac{1}{2}-\delta)}}\right),
\end{align*}
where $C_{r,p,H}$ is a suitable constant depending only on $r$, $p$
and $H$. Hence, if we choose $\delta$ small enough such that
$0<\delta<(1-H)\wedge\left(H-\frac{1}{2}\right)$, then 
\[
\sum_{m=1}^{\infty}c_{p,r}\left(\left\vert B_{t}^{(m+1)}-B_{t}^{(m)}\right\vert >\frac{1}{2^{m\delta}}\right)<\infty,
\]
which implies that $(B_{t}^{(m)})_{m\in\mathbb{N}}$ converges $(p,r)$-quasi-surely
to some random variable $\tilde{B}_{t}$ as $m$ tends to infinity
by the first Borel-Cantelli lemma. One can show that $(B_{t}^{(m)})_{m\in\mathbb{N}}$
converges in $\mathbb{D}_{r}^{p}$ to $B_{t}$ (see e.g. \cite{LI2019141}
for a proof), and
\[
DB_{t}(s)=K(t,s)1_{[0,t]}(s),
\]
with all higher order Malliavin derivatives of $B_{t}$ equal to zero.
Now we can easily prove that there exists a subsequence $(B_{t}^{(m_{k})})_{k\in\mathbb{N}}$
converging $(p,r)$-quasi-surely by choosing this sequence to be such that
(for example by applying H\"{o}lder's inequality) 
\[
\left\Vert \left(B_{t}^{(m_{k+1})}-B_{t}^{(m_{k})}\right)^{2}\right\Vert _{\mathbb{D}_{r}^{p}}\leq\frac{1}{2^{k+1}},
\]
and applying the first Borel-Cantelli lemma as before. If for $\omega\in\boldsymbol{W}$,
there are infinitely many $k$'s such that $\left\vert B_{t}^{(m_{k+1})}(\omega)-B_{t}^{(m_{k})}(\omega)\right\vert >1$,
then $(B_{t}^{(m_{k})}(\omega))_{k\in\mathbb{N}}$ is not Cauchy. Therefore,
by Chebyshev's inequality, 
\begin{align*}
 & \hphantom{=\ \ }\sum_{k=0}^{\infty}c_{p,r}\left\{ \omega:\left\vert B_{t}^{(m_{k+1})}(\omega)-B_{t}^{(m_{k})}(\omega)\right\vert >1\right\} \\
 & =\sum_{k=0}^{\infty}c_{p,r}\left\{ \omega:\left(B_{t}^{(m_{k+1})}(\omega)-B_{t}^{(m_{k})}(\omega)\right)^{2}>1\right\} \\
 & \leq\sum_{k=0}^{\infty}\frac{1}{2^{k+1}}<\infty,
\end{align*}
and hence by the first Borel-Cantelli lemma, 
\[
c_{p,r}\left\{ \left\vert B_{t}^{(m_{k+1})}-B_{t}^{(m_{k})}\right\vert >1\text{ infinitely often}\right\} =0.
\]
As a consequence, $(B_{t}^{(m_{k})})_{k\in\mathbb{N}}$ converges to $B_{t}$
apart from on a slim set, and the uniqueness of limit forces its limit
to be $\tilde{B}_{t}$, which implies $B_{t}=\tilde{B}_{t}$ q.s.
\end{proof}
From now on, we work with the modification of $B$ which is the $(p,r)$-quasi-sure limit of the approximations $B^{(m)}$.

\section{Exponential tightness of the approximation sequence}

For each fixed $t$, $B_{t}$ is quasi-surely defined (with $B_{0}(\omega)=0$
for all $\omega\in\boldsymbol{W}$). We define a map $X^{(m)}:\boldsymbol{W}\to\boldsymbol{W}$
by
\[
X^{(m)}(\omega)(t):=B_{t_{k-1}^{m}}(\omega)+2^{m}\left(t-t_{k-1}^{m}\right)\left(B_{t_{k}^{m}}(\omega)-B_{t_{k-1}^{m}}(\omega)\right),\quad\forall t_{k-1}^{m}\leq t\leq t_{k}^{m},
\]
where $t_{k}^{m}=\frac{k}{2^{m}}$. Then $X^{(m)}$ is $(p,r)$-quasi-surely
defined as it is a linear interpolation of finitely many $B_{t_{k}^{m}}$'s for all $p$ and $r$.
For each $m$, let $X^{\varepsilon,(m)}$ be the scaled map, which
is defined as $X^{\varepsilon,(m)}(\omega)=X^{(m)}(\varepsilon\omega)$.
As $B_{t}$ is the limit of linear combinations of $\omega_{t}$'s,
it follows that $X^{\varepsilon,(m)}(\omega)=\varepsilon X^{(m)}(\omega)$. 

Our goal is to show that the sequence $(X^{(m)})_{m\in\mathbb{N}}$
converges to some $X$ quasi-surely, which implies that $(X^{\varepsilon,(m)})_{m\in\mathbb{N}}$
converges to $X^{\varepsilon}$ quasi-surely, where the scaled map
$X^{\varepsilon}$ is given by $X^{\varepsilon}(\omega)=X(\varepsilon\omega)=\varepsilon X(\omega)$.
Moreover, the fact that the sequence of scaled maps $(X^{\varepsilon,(m)})_{m\in\mathbb{N}}$
converges exponentially fast will be revealed in the proof as well.
Since $X^{\varepsilon}$ is quasi-surely defined with exponentially
good approximations $(X^{\varepsilon,(m)})_{m\in\mathbb{N}}$,
we may apply the result from the LDP theory to conclude the final result. 

We will need the following estimate from the rough path analysis,
which is contained in \cite{Lyons2002} (see Proposition 4.1.1 on
page 62 or equation (4.15) on page 64). Here, we adapt the result
to our case and state it as the following:
\begin{prop}
\label{prop:Lyons Qian}Let $u$ and $w$ be two continuous paths
in a Banach space. Then for any $q>1$ and $\gamma>q-1$, there exists
a constant $C_{q,\gamma}$ depending only on $q$ and $\gamma$ such
that 
\[
\sup_{D}\sum_{l}\left\vert u_{t_{l-1},t_{l}}-w_{t_{l-1},t_{l}}\right\vert ^{q}\leq C_{q,\gamma}\sum_{n=1}^{\infty}n^{\gamma}\sum_{k=1}^{2^{n}}\left\vert u_{t_{k-1}^{n},t_{k}^{n}}-w_{t_{k-1}^{n},t_{k}^{n}}\right\vert ^{q},
\]
where the supremum is taken over all finite partitions $D$ of $[0,1]$,
$t_{l}^{n}=\frac{l}{2^{n}}$ for $n=1,2,\cdots$, $l=0,\cdots,2^{n}$,
and $u_{s,t}=u_{t}-u_{s}$ is the increment of path $u$. 
\end{prop}

Together with Proposition \ref{prop:Lp to L2}, the above estimate
allows us to simplify our problem by controlling the $L^{2}$-norm of Gaussian processes instead of capacities.
\begin{thm}
\label{thm:Exp good X}For $r\in\mathbb{N}$ and $1<p<\infty$, $\left(X^{(m)}\right)_{m\in\mathbb{N}}$
converges $(p,r)$-quasi-surely to some limit $X$, and the scaled
maps $\left(X^{\varepsilon,(m)}\right)_{m\in\mathbb{N}}$ are exponentially good approximations of $\{X^{\varepsilon}:\varepsilon>0\}$ under the
capacity $c_{p,r}$.
\end{thm}

\begin{proof}
Here we use a technique in the theory of rough paths to control the
tails of $X^{(m)}$'s, which are Gaussian. Let us
first prove that the sequence $\left(X^{(m)}\right)_{m\in\mathbb{N}}$ converges
uniformly $(p,r)$-quasi-surely. By using the elementary fact that 
\[
\lVert u-w\rVert\leq\sup_{D}\left(\sum_{l}\left\vert u_{t_{l-1},t_{l}}-w_{t_{l-1},t_{l}}\right\vert ^{q}\right)^{\frac{1}{q}},
\]
for any $u,w\in\boldsymbol{W}$ and for any $q>1$, where the supremum
is taken over all possible finite partitions of $[0,1]$, and $u_{s,t}=u_{t}-u_{s}$,
together with Proposition \ref{prop:Lyons Qian}, we obtain that
\[
\lVert u-w\rVert^{q}\leq C_{q,\gamma}\sum_{n=1}^{\infty}n^{\gamma}\sum_{k=1}^{2^{n}}\left\vert u_{t_{k-1}^{n},t_{k}^{n}}-w_{t_{k-1}^{n},t_{k}^{n}}\right\vert ^{q}
\]
 for $\gamma>q-1$, where $C_{q,\gamma}$ is a constant depending
on $q$ and $\gamma$, and $t_{k}^{n}=\frac{k}{2^{n}}$. We will apply
the above estimate to $X^{(m)}$ to obtain an upper bound of
\begin{equation}
I_{m}(\lambda):=c_{p,r}\left(\left\Vert X^{(m+1)}-X^{(m)}\right\Vert >\lambda\right),\label{eq:17}
\end{equation}
where $\lambda>0$. Since $c_{p,r}$ is increasing in $p$,
we shall assume that $p>2$. By monotonicity and sub-additivity properties
of capacity, we obtain that for $\theta>0$,
\begin{align}
I_{m}(\lambda) & \leq c_{p,r}\left(C_{q,\gamma}\sum_{n=1}^{\infty}n^{\gamma}\sum_{k=1}^{2^{n}}\left\vert X_{t_{k-1}^{n},t_{k}^{n}}^{(m+1)}-X_{t_{k-1}^{n},t_{k}^{n}}^{(m)}\right\vert ^{q}>\lambda^{q}\right)\nonumber \\
 & =c_{p,r}\left(\sum_{n=1}^{\infty}n^{\gamma}\sum_{k=1}^{2^{n}}\left\vert X_{t_{k-1}^{n},t_{k}^{n}}^{(m+1)}-X_{t_{k-1}^{n},t_{k}^{n}}^{(m)}\right\vert ^{q}>C_{q,\gamma}^{-1}C_{\theta,\gamma}\sum_{n=1}^{\infty}n^{\gamma}\frac{\lambda^{q}}{2^{n\theta}}\right)\nonumber \\
 & \leq\sum_{n=1}^{\infty}c_{p,r}\left(\sum_{k=1}^{2^{n}}\left\vert X_{t_{k-1}^{n},t_{k}^{n}}^{(m+1)}-X_{t_{k-1}^{n},t_{k}^{n}}^{(m)}\right\vert ^{q}>C_{q,\gamma}^{-1}C_{\theta,\gamma}\frac{\lambda^{q}}{2^{n\theta}}\right)\nonumber \\
 & \leq\sum_{n=1}^{\infty}\sum_{k=1}^{2^{n}}c_{p,r}\left(\left\vert X_{t_{k-1}^{n},t_{k}^{n}}^{(m+1)}-X_{t_{k-1}^{n},t_{k}^{n}}^{(m)}\right\vert >C_{q,\gamma}^{-\frac{1}{q}}C_{\theta,\gamma}^{\frac{1}{q}}\frac{\lambda}{2^{\frac{n(1+\theta)}{q}}}\right),\label{eq:12}
\end{align}
where $C_{\theta,\gamma}=\left(\sum_{n=1}^{\infty}\frac{n^{\gamma}}{2^{n\theta}}\right)^{-1}$.
We introduce a new parameter $N$, whose value is to be determined
at the end of our proof, and consider 
\begin{equation}
c_{p,r}\left(\left\vert X_{t_{k-1}^{n},t_{k}^{n}}^{(m+1)}-X_{t_{k-1}^{n},t_{k}^{n}}^{(m)}\right\vert ^{2N}>C_{q,\gamma}^{-\frac{2N}{q}}C_{\theta,\gamma}^{\frac{2N}{q}}\frac{\lambda^{2N}}{2^{\frac{2nN(1+\theta)}{q}}}\right).\label{eq:13}
\end{equation}
Notice that when $n\leq m$, $X_{t_{k}^{n}}^{(m)}=B_{t_{2^{m-n}k}^{m}}$.
Since $t_{2^{m-n}k}^{m}=t_{2^{m+1-n}k}^{m+1}$, we have that $X_{t_{k-1}^{n},t_{k}^{n}}^{(m+1)}=X_{t_{k-1}^{n},t_{k}^{n}}^{(m)}$.
By Chebyshev's inequality, we obtain that
\begin{align}
 & \hphantom{=\ \ }c_{p,r}\left(\left\vert X_{t_{k-1}^{n},t_{k}^{n}}^{(m+1)}-X_{t_{k-1}^{n},t_{k}^{n}}^{(m)}\right\vert ^{2N}>C_{q,\gamma,\theta,N}\frac{\lambda^{2N}}{2^{\frac{2nN(1+\theta)}{q}}}\right)\nonumber \\
 & \leq C_{q,\gamma,\theta,N}^{-1}\lambda^{-2N}2^{\frac{2nN(1+\theta)}{q}}\left\Vert \left\vert X_{t_{k-1}^{n},t_{k}^{n}}^{(m+1)}-X_{t_{k-1}^{n},t_{k}^{n}}^{(m)}\right\vert ^{2N}\right\Vert _{\mathbb{D}_{r}^{p}},\label{eq:14}
\end{align}
 where $C_{q,\gamma,\theta,N}=C_{q,\gamma}^{-\frac{2N}{q}}C_{\theta,\gamma}^{\frac{2N}{q}}$
is a constant. Now by Proposition \ref{prop:Lp to L2}, with $\left\vert X_{t_{k-1}^{n},t_{k}^{n}}^{(m+1)}-X_{t_{k-1}^{n},t_{k}^{n}}^{(m)}\right\vert ^{2N}$
a \textcolor{black}{polynomial functional} of degree $2N$, and $N\geq\frac{r}{2}$,
we have that
\begin{align*}
 & \hphantom{=}\ \ \left\Vert \left\Vert D^{l}\left(\left\vert X_{t_{k-1}^{n},t_{k}^{n}}^{(m+1)}-X_{t_{k-1}^{n},t_{k}^{n}}^{(m)}\right\vert ^{2N}\right)\right\Vert _{\mathscr{H}^{\otimes l}}\right\Vert_p \leq(2N+1)(p-1)^{\frac{N}{2}}\left\Vert \left\Vert D^{l}\left(\left\vert X_{t_{k-1}^{n},t_{k}^{n}}^{(m+1)}-X_{t_{k-1}^{n},t_{k}^{n}}^{(m)}\right\vert ^{2N}\right)\right\Vert _{\mathscr{H}^{\otimes l}}\right\Vert _{2},
\end{align*}
and 
\[
\left\Vert \left\Vert D^{l}\left(\left\vert X_{t_{k-1}^{n},t_{k}^{n}}^{(m+1)}-X_{t_{k-1}^{n},t_{k}^{n}}^{(m)}\right\vert ^{2N}\right)\right\Vert _{\mathscr{H}^{\otimes l}}\right\Vert _{2}\leq(2N)^{\frac{r}{2}}\left\Vert \left\vert X_{t_{k-1}^{n},t_{k}^{n}}^{(m+1)}-X_{t_{k-1}^{n},t_{k}^{n}}^{(m)}\right\vert ^{2N}\right\Vert _{2}
\]
for any $0\leq l\leq r$. The above two inequalities imply that
\begin{equation}
\left\Vert \left\vert X_{t_{k-1}^{n},t_{k}^{n}}^{(m+1)}-X_{t_{k-1}^{n},t_{k}^{n}}^{(m)}\right\vert ^{2N}\right\Vert _{\mathbb{D}_{r}^{p}}\leq(r+1)(2N+1)(p-1)^{\frac{N}{2}}(2N)^{\frac{r}{2}}\left\Vert \left\vert X_{t_{k-1}^{n},t_{k}^{n}}^{(m+1)}-X_{t_{k-1}^{n},t_{k}^{n}}^{(m)}\right\vert ^{2N}\right\Vert _{2}.\label{eq:19}
\end{equation}
For $N=1,2,\cdots$, $f(x)=x^{2N}$ is convex, so by Jensen's inequality,
\[
f(x+y)\leq 2^{2N-1}\left(f(x)+f(y)\right),
\]
and hence
\[
\left\vert X_{t_{k-1}^{n},t_{k}^{n}}^{(m+1)}-X_{t_{k-1}^{n},t_{k}^{n}}^{(m)}\right\vert ^{2N}\leq2^{2N-1}\left(\left\vert X_{t_{k-1}^{n},t_{k}^{n}}^{(m+1)}\right\vert ^{2N}+\left\vert X_{t_{k-1}^{n},t_{k}^{n}}^{(m)}\right\vert ^{2N}\right).
\]
Therefore, it suffices to estimate $\left\Vert \left\vert X_{t_{k-1}^{n},t_{k}^{n}}^{(m)}\right\vert ^{2N}\right\Vert _{2}$.
By definition, if $t_{l-1}^{m}\leq t_{k-1}^{n}<t_{k}^{n}\leq t_{l}^{m}$
for some $l$, then
\begin{align*}
X_{t_{k-1}^{n},t_{k}^{n}}^{(m)} & =2^{m}\left(t_{k}^{n}-t_{k-1}^{n}\right)\left(B_{t_{l}^{m}}-B_{t_{l-1}^{m}}\right).
\end{align*}
Hence, by Proposition \ref{prop:Lp to L2},
\begin{align*}
\left\Vert \left\vert X_{t_{k-1}^{n},t_{k}^{n}}^{(m)}\right\vert ^{2N}\right\Vert _{2} & =\left\Vert X_{t_{k-1}^{n},t_{k}^{n}}^{(m)}\right\Vert _{4N}^{2N}\\
 & \leq2^{2N}(4N-1)^{N}\left\Vert X_{t_{k-1}^{n},t_{k}^{n}}^{(m)}\right\Vert _{2}^{2N}\\
 & =2^{2N}(4N-1)^{N}\mathbb{E}\left[\left(2^{m}\left(t_{k}^{n}-t_{k-1}^{n}\right)\left(B_{t_{l}^{m}}-B_{t_{l-1}^{m}}\right)\right)^{2}\right]^{N}\\
\vphantom{\left\Vert X_{t_{k-1}^{n},t_{k}^{n}}^{(m)}\right\Vert _{2}^{2N}} & =2^{2N}(4N-1)^{N}\frac{2^{2mN(1-H)}}{2^{2nN}}.
\end{align*}
%\vphantom{\left\Vert X_{t_{k-1}^{n},t_{k}^{n}}^{(m)}\right\Vert _{2}^{2N}} & =2^{2N}(4N-1)^{N}2^{2mN}\frac{1}{2^{2nN}}\mathbb{E}\left[\left(B_{t_{l}^{m}}-B_{t_{l-1}^{m}}\right)^{2}\right]^{N}\\
It thus implies that 
\begin{align}
\left\Vert \left\vert X_{t_{k-1}^{n},t_{k}^{n}}^{(m+1)}-X_{t_{k-1}^{n},t_{k}^{n}}^{(m)}\right\vert ^{2N}\right\Vert _{2} & \leq2^{2N-1}\left(\left\Vert \left\vert X_{t_{k-1}^{n},t_{k}^{n}}^{(m+1)}\right\vert ^{2N}\right\Vert _{2}+\left\Vert \left\vert X_{t_{k-1}^{n},t_{k}^{n}}^{(m)}\right\vert ^{2N}\right\Vert _{2}\right)\nonumber \\
 & \leq2^{4N-1}(4N-1)^{N}\left(\frac{2^{2(m+1)N(1-H)}}{2^{2nN}}+\frac{2^{2mN(1-H)}}{2^{2nN}}\right)\nonumber \\
 & =C_{N,H}\frac{2^{2mN(1-H)}}{2^{2nN}},\label{eq:20}
\end{align}
where $C_{N,H}=2^{4N-1}\left(4N-1\right)^{N}\left(1+2^{2N(1-H)}\right)$
is a constant depending only on $N$ and $H$. We may conclude from
(\ref{eq:19}) and (\ref{eq:20}) that 
\begin{equation}
\left\Vert \left\vert X_{t_{k-1}^{n},t_{k}^{n}}^{(m+1)}-X_{t_{k-1}^{n},t_{k}^{n}}^{(m)}\right\vert ^{2N}\right\Vert _{\mathbb{D}_{r}^{p}}\leq C_{r,p,N,H}\frac{2^{2mN(1-H)}}{2^{2nN}}\label{eq:15}
\end{equation}
for $n>m$, where 
\begin{equation}
C_{r,p,N,H}=(r+1)(2N+1)(p-1)^{\frac{N}{2}}(2N)^{\frac{r}{2}}C_{N,H}\label{eq:23}
\end{equation}
depends on $r$, $p$, $N$ and $H$. Plugging (\ref{eq:15}) into
(\ref{eq:14}), we obtain that 
\begin{equation}
c_{p,r}\left(\left\vert X_{t_{k-1}^{n},t_{k}^{n}}^{(m+1)}-X_{t_{k-1}^{n},t_{k}^{n}}^{(m)}\right\vert ^{2N}>C_{q,\gamma,\theta}\frac{\lambda^{2N}}{2^{\frac{2nN(1+\theta)}{q}}}\right)\leq C_{q,\gamma,\theta,N}^{-1}C_{r,p,N,H}\lambda^{-2N}\frac{2^{2mN(1-H)}}{2^{2nN\left(1-\frac{1+\theta}{q}\right)}}.\label{eq:21}
\end{equation}
Therefore, according to (\ref{eq:12}), (\ref{eq:13}) and (\ref{eq:21}),
\begin{align*}
I_{m}(\lambda) & \leq\sum_{n=1}^{\infty}\sum_{k=1}^{2^{n}}c_{p,r}\left(\left\vert X_{t_{k-1}^{n},t_{k}^{n}}^{(m+1)}-X_{t_{k-1}^{n},t_{k}^{n}}^{(m)}\right\vert ^{2N}>C_{q,\gamma,\theta,N}\frac{\lambda^{2N}}{2^{\frac{2nN(1+\theta)}{q}}}\right)\\
 & \leq\sum_{n=m+1}^{\infty}\sum_{k=1}^{2^{n}}C_{q,\gamma,\theta,N}^{-1}C_{r,p,N,H}\lambda^{-2N}\frac{2^{2mN(1-H)}}{2^{2nN\left(1-\frac{1+\theta}{q}\right)}}\\
 & =C_{q,\gamma,\theta,N}^{-1}C_{r,p,N,H}\lambda^{-2N}\sum_{n=m+1}^{\infty}\frac{2^{2mN(1-H)}}{2^{n\left(2N\left(1-\frac{1+\theta}{q}\right)-1\right)}}\\
 & =C_{q,\gamma,\theta,N}^{-1}C_{r,p,N,H}\lambda^{-2N}\frac{1}{2^{m\left(2N\left(H-\frac{1+\theta}{q}\right)-1\right)}}\sum_{k=1}^{\infty}\frac{1}{2^{k\left(2N\left(1-\frac{1+\theta}{q}\right)-1\right)}}.
\end{align*}
Since $2N\left(1-\frac{1+\theta}{q}\right)-1>2N\left(H-\frac{1+\theta}{q}\right)-1$,
the above series converges as long as $2N\left(H-\frac{1+\theta}{q}\right)-1>0$,
which means that we need $qH-\frac{q}{2N}-1>0$ for some integer $N$.
Therefore, if we choose $q>\left(H-\frac{1}{2}\right)^{-1}$, and
$\theta\in\left(0,qH-\frac{q}{2N}-1\right)$ then the above series
converges. As a consequence, we thus have
\begin{equation}
I_{m}(\lambda)\leq C_{q,\gamma,\theta,N}^{'}C_{r,p,N,H}\frac{\lambda^{-2N}}{2^{m\left(2N\left(H-\frac{1+\theta}{q}\right)-1\right)}}\label{eq:18}
\end{equation}
for every $m=1,2,\cdots$, where
\begin{align*}
C_{q,\gamma,\theta,N}^{'} & =C_{q,\gamma,\theta,N}^{-1}\sum_{k=1}^{\infty}\frac{1}{2^{k\left(2N\left(1-\frac{1+\theta}{q}\right)-1\right)}}\\
 & \leq C_{q,\gamma,\theta,N}^{-1}\frac{1}{2^{2\left(1-\frac{1+\theta}{q}\right)-1}-1}\\
\vphantom{\frac{1}{2^{2\frac{1+\theta}{q}}}} & =C_{q,\gamma,\theta,N}^{-1}C_{q,\theta}
\end{align*}
 %& =C_{q,\gamma,\theta,N}^{-1}\frac{1}{2^{2N\left(1-\frac{1+\theta}{q}\right)-1}-1}\\
with $C_{q,\theta}=\left(2^{2\left(1-\frac{1+\theta}{q}\right)-1}-1\right)^{-1}$.
From (\ref{eq:18}), we may deduce that 
\begin{equation}
I_{m}(\lambda)\leq C_{q,\gamma,\theta,N}^{-1}C_{q,\theta}C_{r,p,N,H}\frac{\lambda^{-2N}}{2^{m\left(2N\left(H-\frac{1+\theta}{q}\right)-1\right)}}.\label{eq:22}
\end{equation}
Applying the same argument as in the previous theorem, we see that
the problem may be reduced to proving that for a suitable positive $\delta>0$,
\begin{equation}
\sum_{m=1}^{\infty}I_{m}(\frac{1}{2^{m\delta}})<\infty.\label{eq:16}
\end{equation}
Then by the first Borel-Cantelli lemma for capacity, we obtain the
quasi-sure convergence for $\left(X^{(m)}\right)_{m\in\mathbb{N}}$. Since 
\[
I_{m}(\frac{1}{2^{m\delta}})\leq C_{q,\gamma,\theta,N}^{-1}C_{q,\theta}C_{r,p,N,H}\frac{1}{2^{m\left(2N\left(H-\frac{1+\theta}{q}-\delta\right)-1\right)}},
\]
so the series in (\ref{eq:16}) converges as long as we choose $\delta$
such that $\delta<H-\frac{1+\theta}{q}-\frac{1}{2N}$, which must
exist as we have chosen $q$ and $\theta$ such that $2N\left(H-\frac{1+\theta}{q}\right)-1>0$
for some $N\in\mathbb{N}$. Thus, the convergence of the series in (\ref{eq:16})
implies the convergence of $\left(X^{(m)}\right)_{m\in\mathbb{N}}$.
Denote its limit by $X$, then $X$ is defined quasi-surely on $\boldsymbol{W}$. 

Next, we prove the sequence $\left\{ X^{\varepsilon,(m)}:m\geq1,\varepsilon>0\right\} $
converges to $\left\{ X^{\varepsilon}:\varepsilon>0\right\} $ exponentially fast
with respect to the capacity $c_{p,r}$, that is, 
\[
\lim_{m\to\infty}\limsup_{\varepsilon\to0}\varepsilon^{2}\log c_{p,r}\left\{ \omega:\left\Vert X^{\varepsilon,(m)}(\omega)-X^{\varepsilon}(\omega)\right\Vert >\lambda\right\} =-\infty.
\]
To this end, we shall use a similar argument as in the proof above.
By the sub-additivity of capacity, for $\alpha>0$,
\begin{align}
\vphantom{\left(\frac{\lambda C_{\alpha}}{2^{(k-m)\alpha}\varepsilon}\right)} & \hphantom{=}\ \ c_{p,r}\left\{ \omega:\left\Vert X^{\varepsilon,(m)}(\omega),X^{\varepsilon}(\omega)\right\Vert >\lambda\right\} \nonumber \\
\vphantom{\left(\frac{\lambda C_{\alpha}}{2^{(k-m)\alpha}\varepsilon}\right)} & \leq c_{p,r}\left\{ \omega:\sum_{k=m}^{\infty}\left\Vert X^{\varepsilon,(k)}(\omega),X^{\varepsilon,(k+1)}(\omega)\right\Vert >\lambda\right\} \nonumber \\
 & =c_{p,r}\left\{ \omega:\sum_{k=m}^{\infty}\left\Vert X^{\varepsilon,(k)}(\omega),X^{\varepsilon,(k+1)}(\omega)\right\Vert >C_{\alpha}\sum_{k=m}^{\infty}\frac{\lambda}{2^{(k-m)\alpha}}\right\} \nonumber \\
 & \leq\sum_{k=m}^{\infty}c_{p,r}\left\{ \omega:\left\Vert X^{(k)}(\omega),X^{(k+1)}(\omega)\right\Vert >C_{\alpha}\varepsilon^{-1}\frac{\lambda}{2^{(k-m)\alpha}}\right\} \nonumber \\
 & =\sum_{k=m}^{\infty}I_{k}\left(\frac{\lambda C_{\alpha}}{2^{(k-m)\alpha}\varepsilon}\right),\label{eq:26}
\end{align}
where we have used the notations in (\ref{eq:17}), and $C_{\alpha}=\left(\sum_{i=0}^{\infty}\frac{1}{2^{i\alpha}}\right)^{-1}$is
some positive constant depending only on $\alpha$. Recall that up
to now, the only assumption on $N$ is that $N\geq\frac{r}{2}$, and
now we shall pick up a suitable $N$ to show that the convergence of
$\left(X^{\varepsilon,(m)}\right)_{m\in\mathbb{N}}$ is exponentially
fast. By (\ref{eq:22}), 
\begin{align}
I_{k}\left(\frac{\lambda C_{\alpha}}{2^{(k-m)\alpha}\varepsilon}\right) & \leq C_{q,\gamma,\theta,N}^{-1}C_{r,p,N,H}C_{q,\theta}\frac{1}{2^{k\left(2N\left(H-\frac{1+\theta}{q}-\alpha\right)-1\right)}}\frac{1}{2^{2Nm\alpha}}\frac{\varepsilon^{2N}}{\lambda^{2N}C_{\alpha}^{2N}}\nonumber \\
 & =C_{q,\gamma,\theta,N}^{-1}C_{r,p,N,H}C_{q,\theta}C_{\alpha}^{-2N}\frac{1}{2^{k\beta}}\frac{1}{2^{2Nm\alpha}}\frac{\varepsilon^{2N}}{\lambda^{2N}},\label{eq:25}
\end{align}
where $\beta=2N\left(H-\frac{1+\theta}{q}-\alpha\right)-1$. As $C_{q,\gamma,\theta,N}=C_{q,\gamma}^{-\frac{2N}{q}}C_{\theta,\gamma}^{\frac{2N}{q}}$,
where $C_{r,p,N,H}$ is given as in (\ref{eq:23}) with $C_{N,H}=2^{4N-1}\left(4N-1\right)^{N}\left(1+2^{2N(1-H)}\right)$,
we have that
\begin{align}
 & \hphantom{=\ \ }C_{q,\gamma,\theta,N}^{-1}C_{r,p,N,H}\nonumber \\
 & =C_{q,\gamma}^{\frac{2N}{q}}C_{\theta,\gamma}^{-\frac{2N}{q}}(r+1)(2N+1)(p-1)^{\frac{N}{2}}(2N)^{\frac{r}{2}}2^{4N-1}\left(4N-1\right)^{N}\left(1+2^{2N(1-H)}\right)\nonumber \\
 & \leq(r+1)(2N+1)(2N)^{\frac{r}{2}}\left(16C_{q,\gamma}^{\frac{2}{q}}C_{\theta,\gamma}^{-\frac{2}{q}}(p-1)^{\frac{1}{2}}\right)^{N}\left(4N\right)^{N}2^{2N(1-H)}\nonumber \\
 & :=P_{r}(N)C_{q,\gamma,\theta,p,H}^{N}N^{N},\label{eq:24}
\end{align}
% & =(r+1)(2N+1)(2N)^{\frac{r}{2}}\left(64C_{q,\gamma}^{\frac{2}{q}}C_{\theta,\gamma}^{-\frac{2}{q}}(p-1)^{\frac{1}{2}}2^{2(1-H)}\right)^{N}N^{N}\nonumber \\
where 
\[
P_{r}(N)=(r+1)(2N+1)(2N)^{\frac{r}{2}}
\]
is a polynomial of $N$ depending only on $r$, and $C_{q,\gamma,\theta,p,H}=64C_{q,\gamma}^{\frac{2}{q}}C_{\theta,\gamma}^{-\frac{2}{q}}(p-1)^{\frac{1}{2}}2^{2(1-H)}$
is a constant. If we set $\alpha$ such that $\alpha<H-\frac{1+\theta}{q}-\frac{1}{2N}$,
then $\beta>0$. Together with (\ref{eq:25}) and (\ref{eq:24}),
we obtain that
\begin{align*}
\sum_{k=m}^{\infty}I_{k}\left(\frac{\lambda C_{\alpha}}{2^{(k-m)\alpha}\varepsilon}\right) & \leq P_{r}(N)C_{q,\gamma,\theta,p,H}^{N}N^{N}C_{q,\theta}C_{\alpha}^{-2N}\frac{1}{2^{2Nm\alpha}}\frac{\varepsilon^{2N}}{\lambda^{2N}}\sum_{k=m}^{\infty}\frac{1}{2^{k\beta}}\\
 & =P_{r}(N)C_{q,\gamma,\theta,p,H}^{N}N^{N}C_{q,\theta}C_{\alpha}^{-2N}\lambda^{-2N}\varepsilon^{2N}\sum_{k=0}^{\infty}\frac{1}{2^{k\beta}}\frac{1}{2^{m(2N\alpha+\beta)}}\\
 & =P_{r,q,\theta,\beta}(N)C_{q,\gamma,\theta,p,H,\alpha}^{N}\lambda^{-2N}\varepsilon^{2N}N^{N}\frac{1}{2^{m\left(2N\left(H-\frac{1+\theta}{q}\right)-1\right)}},
\end{align*}
where $C_{q,\gamma,\theta,p,H,\alpha}=C_{q,\gamma,\theta,p,H}C_{\alpha}^{-2}$
and $P_{r,q,\theta,\beta}(N)=C_{q,\theta}\left(\sum_{k=0}^{\infty}\frac{1}{2^{k\beta}}\right)P_{r}(N)$.
According to (\ref{eq:26}), it holds that
\begin{align*}
\vphantom{\left(\frac{\varepsilon^{2}N}{\lambda^{2}}\right)} & \hphantom{=}\ \ \varepsilon^{2}\log c_{p,r}\left\{ \omega:\left\Vert X^{\varepsilon,(m)}(\omega),X^{\varepsilon}(\omega)\right\Vert >\lambda\right\} \\
\vphantom{\left(\frac{\varepsilon^{2}N}{\lambda^{2}}\right)} & \leq\varepsilon^{2}\log P_{r,q,\theta,\beta}(N)+\varepsilon^{2}N\log C_{q,\gamma,\theta,p,H,\alpha}\\
 & \hphantom{=}+\varepsilon^{2}N\log\left(\frac{\varepsilon^{2}N}{\lambda^{2}}\right)-\varepsilon^{2}\left(2N\left(H-\frac{1+\theta}{q}\right)-1\right)m\log2.
\end{align*}
For $\varepsilon$ small enough, choose $N=\left\lfloor \varepsilon^{-2}\right\rfloor $.
Then since $P_{r,q,\theta,\beta}(N)$ is a polynomial of $N$, it
holds that 
\[
\limsup_{\varepsilon\to0}\varepsilon^{2}\log c_{p,r}\left\{ \omega:\left\Vert X^{\varepsilon,(m)}(\omega),X^{\varepsilon}(\omega)\right\Vert >\lambda\right\} \leq\log C-2\left(H-\frac{1+\theta}{q}\right)m\log2,
\]
where $C=C_{q,\gamma,\theta,p,H,\alpha}\lambda^{-2}$ is a constant.
Therefore, as $H>\frac{1+\theta}{q}$, 
\[
\lim_{m\to\infty}\limsup_{\varepsilon\to0}\varepsilon^{2}\log c_{p,r}\left\{ \omega:\left\Vert X^{\varepsilon,m}(\omega),X^{\varepsilon}(\omega)\right\Vert >\lambda\right\} =-\infty,
\]
which completes the proof.
\end{proof}

\section{The proof of the main result}

This section is devoted to the proof of the large deviation principles
stated in the second part of the main result, Theorem \ref{thm:main LDP}.

Notice that for each $m$, $X^{(m)}$, which is a Wiener functional
on $\boldsymbol{W}$ defined $(p,r)$-quasi-surely, is a linear interpolation
of some Gaussian random variables, so we may consider $F_{m}:\mathbb{R}^{2^{m}+1}\to\boldsymbol{W}$,
where 
\begin{equation}
F_{m}(x_{0},\cdots x_{2^{m}})(t)=x_{k-1}+2^{m}\left(t-\frac{k-1}{2^{m}}\right)\left(x_{k}-x_{k-1}\right),\quad\forall t\in\left[\frac{k-1}{2^{m}},\frac{k}{2^{m}}\right],\label{eq:31}
\end{equation}
which maps a $(2^{m}+1)$-dimensional vector to its linear interpolation.
Let us apply Varadhan's contraction principle to the maps above. As
the rate function for the vector-valued Gaussian random variable $\left(B_{0},B_{\frac{1}{2^{m}}},\cdots,B_{\frac{k}{2^{m}}},\cdots,B_{1}\right)$
is computable, the quasi-sure version of LDP may be established easily
for $X^{(m)}$. 
\begin{prop}
Let $\boldsymbol{t}:=\left\{ 0\leq t_{1}<t_{2}<\cdots<t_{n}\leq1\right\} $
be a finite partition of $[0,1]$. Define $T^{\varepsilon}:\boldsymbol{W}\to\mathbb{R}^{n}$
(for $\varepsilon>0$) by $T^{\varepsilon}\left(\omega\right)=\boldsymbol{B}_{\boldsymbol{t}}(\varepsilon\omega)$,
where
\[
\boldsymbol{B}_{\boldsymbol{t}}(\omega)=\left(B_{t_{1}}(\omega),\cdots,B_{t_{n}}(\omega)\right)
\]
is a Gaussian vector with covariance matrix $\boldsymbol{\Sigma}=\left(\sigma_{ij}\right)_{1\leq i,j\leq n}$
and $\sigma_{ij}=\text{Cov}(t_{i},t_{j})$. Then $\left\{ T^{\varepsilon}:\varepsilon>0\right\} $
satisfies $c_{p,r}$-LDP with the good rate function $I_{n}:\mathbb{R}^{n}\to[0,\infty]$
given by 
\[
I_{n}(\boldsymbol{x})=\frac{1}{2}\boldsymbol{x}^{T}\boldsymbol{\Sigma}^{-1}\boldsymbol{x}.
\]
\end{prop}

\begin{proof}
According to Definition \ref{def:LDP}, we need to establish the upper
bound and the lower bound. Since $(p,r)$-capacity is increasing in $p$
and $r$, the lower bound part follows directly from the classical
LDPs for Gaussian measures, to conclude that for all open $G\subset\W$,
\begin{align*}
\limsup_{\varepsilon\to0}\varepsilon^{2}\log c_{p,r}\left\{ \omega\in\boldsymbol{W}:T^{\varepsilon}(\omega)\in G\right\}  & \geq\frac{1}{p}\limsup_{\varepsilon\to0}\varepsilon^{2}\log\mathbb{P}\left\{ \omega\in\boldsymbol{W}:T^{\varepsilon}(\omega)\in G\right\} \\
 & \geq-\frac{1}{p}\inf_{y\in G}I_{n}(y).
\end{align*}

For the upper bound part, we first establish the result when $n=1$.
Let $a>0$. By Chebyshev's inequality, for all $\lambda>0$,
\begin{align*}
c_{p,r}\left\{ \omega:B_{t}(\varepsilon\omega)>a\right\}  & =c_{p,r}\left\{ \omega:e^{\lambda\varepsilon B_{t}(\omega)}>e^{\lambda a}\right\} \vphantom{\sum_{a}^{b}}\\
 & \leq e^{-\lambda a}\left\Vert e^{\lambda\varepsilon B_{t}}\right\Vert _{\mathbb{D}_{r}^{p}}\vphantom{\sum_{a}^{b}}\\
 & =e^{-\lambda a}\left(\sum_{l=0}^{r}\mathbb{E}\left[\left\vert \left\Vert D^{l}\left(e^{\lambda\varepsilon B_{t}}\right)\right\Vert _{\mathscr{H}^{\otimes l}}\right\vert ^{p}\right]\right)^{\frac{1}{p}}.
\end{align*}
Recall that
\[
D\left(e^{\lambda\varepsilon B_{t}}\right)(s)=\lambda\varepsilon e^{\lambda\varepsilon B_{t}}K(t,s)\mathds{1}_{(0,t)}(s),
\]
so by iteration, 
\[
D^{l}\left(e^{\lambda\varepsilon B_{t}}\right)(s_{1},s_{2},\cdots,s_{l})=\left(\lambda\varepsilon\right)^{l}e^{\lambda\varepsilon B_{t}}\left(K(t)\mathds{1}_{(0,t)}\right)^{\otimes l}(s_{1},s_{2},\cdots,s_{l})
\]
for all $l\leq r$, where $\left(K(t)\mathds{1}_{(0,t)}\right)^{\otimes l}$
denotes the $l$-fold tensor product of $K(t,s)\mathds{1}_{(0,t)}(s)$
with itself. Therefore,
\begin{align*}
\left\Vert D^{l}\left(e^{\lambda\varepsilon B_{t}}\right)\right\Vert _{\mathscr{H}^{\otimes l}}^{2} & =\left(\lambda\varepsilon\right)^{2l}e^{2\lambda\varepsilon B_{t}}\left(\int_{0}^{1}K^{2}(t,s)\mathds{1}_{(0,t)}(s)ds\right)^{l}\\
 & =\left(\lambda\varepsilon\right)^{2l}e^{2\lambda\varepsilon B_{t}}t^{2Hl},
\end{align*}
and hence
\begin{align*}
\mathbb{E}\left[\left\vert \left\Vert D^{l}\left(e^{\lambda\varepsilon B_{t}}\right)\right\Vert _{\mathscr{H}^{\otimes l}}\right\vert ^{p}\right] & =\left(\lambda\varepsilon\right)^{lp}t^{Hlp}\mathbb{E}\left[e^{\lambda\varepsilon pB_{t}}\right]\\
 & =\left(\lambda\varepsilon\right)^{lp}t^{Hlp}e^{\frac{(\lambda\varepsilon p)^{2}t^{2H}}{2}}.
\end{align*}
It thus follows that
\begin{align*}
c_{p,r}\left\{ \omega:B_{t}(\varepsilon\omega)>a\right\}  & \leq e^{-\lambda a}\left(\sum_{l=0}^{r}\mathbb{E}\left[\left\vert \left\Vert D^{l}\left(e^{\lambda\varepsilon B_{t}}\right)\right\Vert _{\mathscr{H}^{\otimes l}}\right\vert ^{p}\right]\right)^{\frac{1}{p}}\\
 & \leq e^{-\lambda a}\sum_{l=0}^{r}\mathbb{E}\left[\left\vert \left\Vert D^{l}\left(e^{\lambda\varepsilon B_{t}}\right)\right\Vert _{\mathscr{H}^{\otimes l}}\right\vert ^{p}\right]^{\frac{1}{p}}\\
 & =e^{\frac{(\lambda\varepsilon)^{2}pt^{2H}}{2}-\lambda a}\sum_{l=0}^{r}\left(\lambda\varepsilon t^{H}\right)^{l},
\end{align*}
so that
\begin{equation}
\varepsilon^{2}\log c_{p,r}\left\{ \omega:B_{t}(\varepsilon\omega)>a\right\} \leq\frac{\lambda^{2}\varepsilon^{4}pt^{2H}}{2}-\lambda a\varepsilon^{2}+\varepsilon^{2}\log\left(\sum_{l=0}^{r}\left(\lambda\varepsilon t^{H}\right)^{l}\right).\label{eq:29}
\end{equation}
Setting $\lambda=\frac{a}{p\varepsilon^{2}t^{2H}}$ so that the sum
of first two terms in (\ref{eq:29}) attains its minimum, we obtain
that 
\begin{align*}
\varepsilon^{2}\log c_{p,r}\left\{ \omega:B_{t}(\varepsilon\omega)>a\right\}  & \leq-\frac{a^{2}}{2pt^{2H}}+\varepsilon^{2}\log\left((r+1)\max_{0\leq l\leq r}\left(\frac{a}{\varepsilon pt^{H}}\right)^{l}\right)\\
 & =-\frac{a^{2}}{2pt^{2H}}+\varepsilon^{2}\log(r+1)+\max_{0\leq l\leq r}l\varepsilon^{2}\log\left(\frac{a}{\varepsilon pt^{H}}\right).
\end{align*}
It follows that 
\begin{equation}
\limsup_{\varepsilon\to0}\varepsilon^{2}\log c_{p,r}\left\{ \omega:B_{t}(\varepsilon\omega)>a\right\} \leq-\frac{1}{2p}\cdot\frac{a^{2}}{t^{2H}}=-\frac{1}{p}\inf_{x>a}I_{1}(x),\label{eq:30}
\end{equation}
which remains true if we replace $\left\{ \omega:B_{t}(\varepsilon\omega)>a\right\} $
with $\left\{ \omega:B_{t}(\varepsilon\omega)\geq a\right\} $. We
may deduce the similar results for $\left\{ \omega:B_{t}(\varepsilon\omega)<b\right\} $
and $\left\{ \omega:B_{t}(\varepsilon\omega)\leq b\right\} $ with
$b<0$ by symmetry. 

Now deal the case of a finite partition $\boldsymbol{t}=\left\{ 0\leq t_{1}\leq\cdots\leq t_{n}\leq1\right\} $.
Then $\boldsymbol{B_{t}}(\varepsilon\omega)=\varepsilon\boldsymbol{B_{t}}(\omega)$.
Introduce an inner product $\langle\cdot,\cdot\rangle_{\boldsymbol{\Sigma}}$
on $\mathbb{R}^{n}$: 
\[
\langle\boldsymbol{x},\boldsymbol{y}\rangle_{\boldsymbol{\Sigma}}=\boldsymbol{x}^{T}\boldsymbol{\Sigma}^{-1}\boldsymbol{y},
\]
and denote the corresponding norm by $\lvert\cdot\rvert_{\boldsymbol{\Sigma}}$.
Notice that for any $\boldsymbol{x}=(x_{1},\cdots,x_{n})\in B(\boldsymbol{a},r)$,
the open ball in $\left(\mathbb{R}^{n},\lvert\cdot\rvert_{\boldsymbol{\Sigma}}\right)$
with centre $\boldsymbol{a}$ and radius $r$, 
\[
\langle\boldsymbol{\lambda},\boldsymbol{a}-\boldsymbol{x}\rangle_{\boldsymbol{\Sigma}}\leq\vert\boldsymbol{\lambda}\vert_{\Sgm}\vert\boldsymbol{a}-\boldsymbol{x}\vert_{\Sgm}\leq r\vert\boldsymbol{\lambda}\vert_{\Sgm}
\]
for all $\boldsymbol{\lambda}\in\mathbb{R}^{n}$, which implies that
$B(\boldsymbol{a},r)\subset\left\{ \boldsymbol{x}:\langle\boldsymbol{\lambda},\boldsymbol{a}-\boldsymbol{x}\rangle_{\Sgm}\leq r\vert\boldsymbol{\lambda}\vert_{\Sgm}\right\} $.
Based on this observation, we may apply Chebyshev's inequality and
get that
\begin{align}
c_{p,r}\left\{ \omega:\boldsymbol{B}_{\boldsymbol{t}}(\varepsilon\omega)\in B(\boldsymbol{a},r)\right\}  & \leq c_{p,r}\left\{ \omega:\langle\boldsymbol{\lambda},\boldsymbol{a}-\boldsymbol{B}_{\boldsymbol{t}}(\varepsilon\omega)\rangle_{\Sgm}\leq r\vert\boldsymbol{\lambda}\vert_{\Sgm}\right\} \nonumber \\
 & =c_{p,r}\left\{ \omega:e^{\langle\boldsymbol{\lambda},\boldsymbol{B}_{\boldsymbol{t}}(\varepsilon\omega)\rangle_{\Sgm}}\geq e^{\langle\boldsymbol{\lambda},\boldsymbol{a}\rangle_{\Sgm}-r\vert\boldsymbol{\lambda}\vert_{\Sgm}}\right\} \nonumber \\
 & \leq e^{r\vert\boldsymbol{\lambda}\vert_{\Sgm}-\langle\boldsymbol{\lambda},\boldsymbol{a}\rangle_{\Sgm}}\left\Vert e^{\langle\boldsymbol{\lambda},\boldsymbol{B}_{\boldsymbol{t}}(\varepsilon\omega)\rangle_{\Sgm}}\right\Vert _{\mathbb{D}_{r}^{p}}\label{eq:27}
\end{align}
for all $\boldsymbol{\lambda}=\left(\lambda_{1},\cdots,\lambda_{n}\right)\in\mathbb{R}^{n}$.
By the Chain rule for Malliavin derivatives, 
\[
D\left(e^{\langle\boldsymbol{\lambda},\boldsymbol{B}_{\boldsymbol{t}}(\varepsilon\omega)\rangle_{\Sgm}}\right)(s)=\varepsilon e^{\langle\boldsymbol{\lambda},\boldsymbol{B}_{\boldsymbol{t}}(\varepsilon\omega)\rangle_{\Sgm}}\langle\boldsymbol{\lambda},D\boldsymbol{B}_{\boldsymbol{t}}\rangle_{\Sgm}(s),
\]
where $D\boldsymbol{B}_{\boldsymbol{t}}=\left(DB_{t_{1}},\cdots,DB_{t_{n}}\right)$,
and by iteration, 
\[
D^{l}\left(e^{\langle\boldsymbol{\lambda},\boldsymbol{B}_{\boldsymbol{t}}(\varepsilon\omega)\rangle_{\Sgm}}\right)(s_{1},\cdots,s_{l})=\varepsilon^{l}e^{\langle\boldsymbol{\lambda},\boldsymbol{B}_{\boldsymbol{t}}(\varepsilon\omega)\rangle_{\Sgm}}\langle\boldsymbol{\lambda},D\boldsymbol{B}_{\boldsymbol{t}}\rangle_{\Sgm}^{\otimes l}(s_{1},\cdots,s_{l}),
\]
where $\langle\boldsymbol{\lambda},D\boldsymbol{B}_{\boldsymbol{t}}\rangle_{\Sgm}^{\otimes l}$
denotes the $l$-fold tensor product of $\langle\boldsymbol{\lambda},D\boldsymbol{B}_{\boldsymbol{t}}\rangle_{\Sgm}$
with itself. Since $\langle DB_{t_{i}},DB_{t_{j}}\rangle_{\mathscr{H}}=\sigma_{ij}$
for all $1\leq i,j\leq n$, it follows that
\begin{align*}
\left\Vert D^{l}\left(e^{\langle\boldsymbol{\lambda},\boldsymbol{B}_{\boldsymbol{t}}(\varepsilon\omega)\rangle_{\Sgm}}\right)\right\Vert _{\mathscr{H}^{\otimes l}}^{2} & =\varepsilon^{2l}e^{2\varepsilon\langle\boldsymbol{\lambda},\boldsymbol{B_{t}}(\omega)\rangle_{\Sgm}}\langle\langle\boldsymbol{\lambda},D\boldsymbol{B}_{\boldsymbol{t}}\rangle_{\Sgm},\langle\boldsymbol{\lambda},D\boldsymbol{B}_{\boldsymbol{t}}\rangle_{\Sgm}\rangle_{\mathscr{H}}^{l}\vphantom{\sum_{a}^{b}}\\
 & =\varepsilon^{2l}e^{2\varepsilon\langle\boldsymbol{\lambda},\boldsymbol{B_{t}}(\omega)\rangle_{\Sgm}}\left\langle \sum_{i=1}^{n}\left(\boldsymbol{\lambda}^{T}\Sgm^{-1}\right)_{i}DB_{t_{i}},\sum_{j=1}^{n}\left(\boldsymbol{\lambda}^{T}\Sgm^{-1}\right)_{j}DB_{t_{j}}\right\rangle _{\mathscr{H}}^{l}\\
 & =\varepsilon^{2l}e^{2\varepsilon\langle\boldsymbol{\lambda},\boldsymbol{B_{t}}(\omega)\rangle_{\Sgm}}\left(\sum_{1\leq i,j\leq n}\left(\boldsymbol{\lambda}^{T}\Sgm^{-1}\right)_{i}\left(\boldsymbol{\lambda}^{T}\Sgm^{-1}\right)_{j}\sigma_{ij}\right)^{l}\\
 & =\varepsilon^{2l}\vert\boldsymbol{\lambda}\vert_{\Sgm}^{2l}e^{2\varepsilon\langle\boldsymbol{\lambda},\boldsymbol{B_{t}}(\omega)\rangle_{\Sgm}},\vphantom{\sum_{a}^{b}}
\end{align*}
where $\left(\boldsymbol{\lambda}^{T}\Sgm^{-1}\right)_{i}$ denotes
the $i$-th component of $\boldsymbol{\lambda}^{T}\Sgm^{-1}$. Thus,
\begin{align*}
\mathbb{E}\left[\left\vert \left\Vert D^{l}\left(e^{\langle\boldsymbol{\lambda},\boldsymbol{B}_{\boldsymbol{t}}(\varepsilon\omega)\rangle_{\Sgm}}\right)\right\Vert _{\mathscr{H}^{\otimes l}}\right\vert ^{p}\right] & =\varepsilon^{lp}\vert\boldsymbol{\lambda}\vert_{\Sgm}^{lp}\mathbb{E}\left[e^{\varepsilon p\langle\boldsymbol{\lambda},\boldsymbol{B_{t}}(\omega)\rangle_{\Sgm}}\right]\\
 & =\varepsilon^{lp}\vert\boldsymbol{\lambda}\vert_{\Sgm}^{lp}e^{\frac{1}{2}(\varepsilon p)^{2}\left(\Sgm^{-1}\boldsymbol{\lambda}\right)^{T}\Sgm\left(\Sgm^{-1}\boldsymbol{\lambda}\right)}\\
 & =\varepsilon^{lp}\vert\boldsymbol{\lambda}\vert_{\Sgm}^{lp}e^{\frac{1}{2}(\varepsilon p)^{2}\vert\boldsymbol{\lambda}\vert_{\Sgm}^{2}},
\end{align*}
which implies that
\begin{align*}
\left\Vert e^{\langle\boldsymbol{\lambda},\boldsymbol{B}_{\boldsymbol{t}}(\varepsilon\omega)\rangle}\right\Vert _{\mathbb{D}_{r}^{p}} & \leq\sum_{l=0}^{r}\mathbb{E}\left[\left\vert \left\Vert D^{l}\left(e^{\langle\boldsymbol{\lambda},\boldsymbol{B}_{\boldsymbol{t}}(\varepsilon\omega)\rangle}\right)\right\Vert _{\mathscr{H}^{\otimes l}}\right\vert ^{p}\right]^{\frac{1}{p}}\\
 & =\sum_{l=0}^{r}\varepsilon^{l}\vert\boldsymbol{\lambda}\vert_{\Sgm}^{l}e^{\frac{1}{2}\varepsilon^{2}p\vert\boldsymbol{\lambda}\vert_{\Sgm}^{2}}.
\end{align*}
Therefore, by (\ref{eq:27}), 
\[
c_{p,r}\left\{ \omega:\boldsymbol{B}_{\boldsymbol{t}}(\varepsilon\omega)\in B(\boldsymbol{a},r)\right\} \leq\sum_{l=0}^{r}\left(\varepsilon\vert\boldsymbol{\lambda}\rvert_{\Sgm}\right)^{l}e^{\frac{1}{2}\varepsilon^{2}p\vert\boldsymbol{\lambda}\vert_{\Sgm}^{2}+r\vert\boldsymbol{\lambda}\vert_{\Sgm}-\langle\boldsymbol{\lambda},\boldsymbol{a}\rangle_{\Sgm}}.
\]
As a consequence, we have that
\begin{align}
\varepsilon^{2}\log c_{p,r}\left\{ \omega:\boldsymbol{B}_{\boldsymbol{t}}(\varepsilon\omega)\in B(\boldsymbol{a},r)\right\}  & \leq\frac{1}{2}\varepsilon^{4}p\vert\boldsymbol{\lambda}\vert_{\Sgm}^{2}+\varepsilon^{2}r\vert\boldsymbol{\lambda}\vert_{\Sgm}-\varepsilon^{2}\langle\boldsymbol{\lambda},\boldsymbol{a}\rangle_{\Sgm}\nonumber \\
 & \hphantom{=}+\varepsilon^{2}\log\left(\sum_{l=0}^{r}\left(\varepsilon\vert\boldsymbol{\lambda}\rvert_{\Sgm}\right)^{l}\right).\label{eq:28}
\end{align}
Choose $\boldsymbol{\lambda}$ such that  
\[
f(\boldsymbol{\lambda})=\frac{1}{2}\varepsilon^{4}p\vert\boldsymbol{\lambda}\vert_{\Sgm}^{2}+\varepsilon^{2}r\vert\boldsymbol{\lambda}\vert_{\Sgm}-\varepsilon^{2}\langle\boldsymbol{\lambda},\boldsymbol{a}\rangle_{\Sgm}
\]
attains its minimum, which happens when $\boldsymbol{\lambda}$ has
the same direction as $\boldsymbol{a}$ since the first two terms
only depends on the magnitude of $\boldsymbol{\lambda}$, i.e. we may
write $\boldsymbol{\lambda}=\boldsymbol{a}\vert\boldsymbol{\lambda}\vert_{\Sgm}\vert\boldsymbol{a}\vert_{\Sgm}^{-1}$.
Then the function becomes 
\[
f(\boldsymbol{\lambda})=\frac{1}{2}\varepsilon^{4}p\vert\boldsymbol{\lambda}\vert_{\Sgm}^{2}+\varepsilon^{2}r\vert\boldsymbol{\lambda}\vert_{\Sgm}-\varepsilon^{2}\vert\boldsymbol{a}\vert_{\Sgm}\vert\boldsymbol{\lambda}\vert_{\Sgm},
\]
which is a quadratic function of $\vert\boldsymbol{\lambda}\vert_{\Sgm}$,
we thus deduce that it reaches its minimum when 
\[
\vert\boldsymbol{\lambda}\vert_{\Sgm}=\frac{\left(\vert\boldsymbol{a}\vert_{\Sgm}-r\right)^{+}}{\varepsilon^{2}p}.
\]
Therefore, the minimum is attained at
\[
\boldsymbol{\lambda}=\frac{\left(\vert\boldsymbol{a}\vert_{\Sgm}-r\right)^{+}}{\varepsilon^{2}p\vert\boldsymbol{a}\vert_{\Sgm}}\boldsymbol{a}.
\]
By setting $\boldsymbol{\lambda}$ equal to the above value in (\ref{eq:28}),
we obtain that 
\begin{align*}
\varepsilon^{2}\log c_{p,r}\left\{ \omega:\boldsymbol{B}_{\boldsymbol{t}}(\varepsilon\omega)\in B(\boldsymbol{a},r)\right\}  & \leq-\frac{1}{2p}\left(\left(\vert\boldsymbol{a}\vert_{\Sgm}-r\right)^{+}\right)^{2}+\varepsilon^{2}\log\left(\sum_{l=0}^{r}\left(\frac{\left(\vert\boldsymbol{a}\vert_{\Sgm}-r\right)^{+}}{\varepsilon p}\right)^{l}\right)\\
\vphantom{\left(\frac{\left(\vert\boldsymbol{a}\vert_{\Sgm}-r\right)^{+}}{\varepsilon p}\right)} & \leq-\frac{1}{2p}\left(\left(\vert\boldsymbol{a}\vert_{\Sgm}-r\right)^{+}\right)^{2}+\varepsilon^{2}\log\left(r+1\right)\\
 & \hphantom{=}+\max_{0\leq l\leq r}\varepsilon^{2}l\log\left(\frac{\left(\vert\boldsymbol{a}\vert_{\Sgm}-r\right)^{+}}{\varepsilon p}\right),
\end{align*}
which implies that 
\begin{align*}
\limsup_{\varepsilon\to0}\varepsilon^{2}\log c_{p,r}\left\{ \omega:\boldsymbol{B}_{\boldsymbol{t}}(\varepsilon\omega)\in B(\boldsymbol{a},r)\right\}  & \leq-\frac{1}{2p}\left(\left(\vert\boldsymbol{a}\vert_{\Sgm}-r\right)^{+}\right)^{2}\\
 & =-\frac{1}{p}\inf_{\boldsymbol{x}\in B(\boldsymbol{a},r)}I_{n}(\boldsymbol{x}).
\end{align*}

Now for any compact $K\subset(\mathbb{R}^{n},\lvert\cdot\rvert_{\Sgm})$
and any $\delta>0$, there exists a finite open cover $\{B(\boldsymbol{a}_{i},\delta)\}_{i\in I}$
in $(\mathbb{R}^{n},\lvert\cdot\rvert_{\Sgm})$ of $K$ with $\boldsymbol{a}_{i}\in K$
and $I$ a finite index set. Therefore, 
\begin{align*}
 & \hphantom{=\ \ }\limsup_{\varepsilon\to0}\varepsilon^{2}\log c_{p,r}\left\{ \omega:\boldsymbol{B}_{\boldsymbol{t}}(\varepsilon\omega)\in K\right\} \\
 & \leq\limsup_{\varepsilon\to0}\varepsilon^{2}\log c_{p,r}\left\{ \omega:\boldsymbol{B}_{\boldsymbol{t}}(\varepsilon\omega)\in\bigcup_{i\in I}B(\boldsymbol{a}_{i},\delta)\right\} \\
 & \leq\limsup_{\varepsilon\to0}\varepsilon^{2}\log\left(\sum_{i\in I}c_{p,r}\left\{ \omega:\boldsymbol{B}_{\boldsymbol{t}}(\varepsilon\omega)\in B(\boldsymbol{a}_{i},\delta)\right\} \right)\\
\vphantom{\sum_{i\in I}} & \leq\limsup_{\varepsilon\to0}\varepsilon^{2}\log\vert I\vert+\limsup_{\varepsilon\to0}\varepsilon^{2}\log\left(\max_{i\in I}c_{p,r}\left\{ \omega:\boldsymbol{B}_{\boldsymbol{t}}(\varepsilon\omega)\in B(\boldsymbol{a}_{i},\delta)\right\} \right)\\
\vphantom{\sum_{i\in I}} & =\max_{i\in I}\limsup_{\varepsilon\to0}\varepsilon^{2}\log c_{p,r}\left\{ \omega:\boldsymbol{B}_{\boldsymbol{t}}(\varepsilon\omega)\in B(\boldsymbol{a}_{i},\delta)\right\} \\
 & \leq\max_{i\in I}-\frac{1}{p}\inf_{\boldsymbol{x}\in B(\boldsymbol{a}_{i},\delta)}I_{n}(\boldsymbol{x})\\
 & =-\frac{1}{p}\min_{i\in I}\inf_{\boldsymbol{x}\in B(\boldsymbol{a}_{i},\delta)}I_{n}(\boldsymbol{x})\\
 & \leq-\frac{1}{p}\inf_{\boldsymbol{x}\in B(K,\delta)}I_{n}(\boldsymbol{x}),
\end{align*}
where $B(K,\delta)=\left\{ \boldsymbol{x}\in\mathbb{R}^{n}:\inf_{y\in K}\lvert\boldsymbol{x}-\boldsymbol{y}\rvert_{\Sgm}<\delta\right\} $.
Let $\delta\to0$, then the upper bound is established for all compact sets.

Now for any $F\subset\mathbb{R}^{n}$ closed under Euclidean metric,
as all norms on $\mathbb{R}^{n}$ are equivalent, $F$ is also closed
in $(\mathbb{R}^{n},\lvert\cdot\rvert_{\Sgm})$. For $\rho>0$, let
$H_{\rho}=\left\{ \boldsymbol{x}=(x_{1},\cdots,x_{n}):\vert x_{i}\vert\leq\rho,\forall1\leq i\leq n\right\} $
be a hypercube in $\mathbb{R}^{n}$. Then by sub-additivity property,
\begin{align*}
c_{p,r}\left\{ \omega:\boldsymbol{B}_{\boldsymbol{t}}(\varepsilon\omega)\in F\right\}  & \leq c_{p,r}\left\{ \omega:\boldsymbol{B}_{\boldsymbol{t}}(\varepsilon\omega)\in F\cap H_{\rho}\right\} +c_{p,r}\left\{ \omega:\boldsymbol{B}_{\boldsymbol{t}}(\varepsilon\omega)\in H_{\rho}^{C}\right\} \\
 & \leq c_{p,r}\left\{ \omega:\boldsymbol{B}_{\boldsymbol{t}}(\varepsilon\omega)\in F\cap H_{\rho}\right\} +\sum_{i=1}^{n}c_{p,r}\left\{ \omega:\left\vert B_{t_{i}}(\varepsilon\omega)\right\vert >\rho\right\} .
\end{align*}
Therefore, by the result for compact sets and (\ref{eq:30}), as well
as Lemma 1.2.15 on page 7 in \cite{Dembo2009}, we have that
\begin{align*}
 & \hphantom{=\ \ }\limsup_{\varepsilon\to0}\varepsilon^{2}\log c_{p,r}\left\{ \omega:\boldsymbol{B}_{\boldsymbol{t}}(\varepsilon\omega)\in F\right\} \\
 & \leq\max\left\{ \limsup_{\varepsilon\to0}\varepsilon^{2}\log c_{p,r}\left\{ \omega:\boldsymbol{B}_{\boldsymbol{t}}(\varepsilon\omega)\in F\cap H_{\rho}\right\} ,\right.\\
 & \hphantom{=aaa}\left.\limsup_{\varepsilon\to0}\varepsilon^{2}\log\left(\sum_{i=1}^{n}c_{p,r}\left\{ \omega:\left\vert B_{t_{i}}(\varepsilon\omega)\right\vert >\rho\right\} \right)\right\} \\
 & \leq\max\left\{ \limsup_{\varepsilon\to0}\varepsilon^{2}\log c_{p,r}\left\{ \omega:\boldsymbol{B}_{\boldsymbol{t}}(\varepsilon\omega)\in F\cap H_{\rho}\right\} ,\right.\\
 & \hphantom{=aaa}\left.\limsup_{\varepsilon\to0}\varepsilon^{2}\log\left(c_{p,r}\left\{ \omega:\left\vert B_{t_{i}}(\varepsilon\omega)\right\vert >\rho\right\} \right)\right\} \\
 & \leq\max\left\{ -\frac{1}{p}\inf_{\boldsymbol{x}\in F\cap H_{\rho}}I_{n}(\boldsymbol{x}),-\frac{1}{p}\inf_{x>\rho}I_{1}(x)\right\} 
\end{align*}
for all $\rho>0$. The proof is complete by letting $\rho\to\infty$.
\end{proof}
Now we may conclude our proof of the quasi-sure large deviation principle, the second part of Theorem \ref{thm:main LDP}. 
\begin{proof}[Proof of Theorem \ref{thm:main LDP}, (2)]
Since $F_{m}:\mathbb{R}^{2^{m}+1}\to\boldsymbol{W}$ defined in (\ref{eq:31})
is continuous and $F_{m}\circ T^{\varepsilon}=X^{\varepsilon,(m)}$,
by the contraction principle (Theorem \ref{thm:Contraction-Principle}),
the family $\left\{ X^{\varepsilon,(m)}\right\} $ satisfies the $c_{p,r}$-LDP
with the good rate function 
\[
J_{m}(\omega)=\inf_{\boldsymbol{x}:F_{m}(\boldsymbol{x})=\omega}I_{2^{m}+1}(\boldsymbol{x}),\quad\omega\in\boldsymbol{W},
\]
where we define $\inf\emptyset=\infty$. When $p=1$ and $r=0$, the
capacity $c_{p,r}$ coincides with Wiener measure $\P$, and we would
expect that the classical LDP for fBM defined on the classical Wiener
space holds. 

Now define $\hat{F}_{m}:\W\to\W$ by 
\begin{equation}
\hat{F}_{m}(\omega)=\omega\left(\frac{k-1}{2^{m}}\right)+2^{m}\left(t-\frac{k-1}{2^{m}}\right)\left(\omega\left(\frac{k}{2^{m}}\right)-\omega\left(\frac{k-1}{2^{m}}\right)\right),\quad t\in\left[\frac{k-1}{2^{m}},\frac{k}{2^{m}}\right],\label{eq:37}
\end{equation}
which is a continuous mapping with respect to the uniform convergence
topology. Then by the classical contraction principle for measures
(see Theorem 4.2.1 on page 126 in \cite{Dembo2009}), the family
$\left\{ \Q_{\varepsilon}\circ\hat{F}_{m}^{-1}\right\} $ on $(\W,\mathscr{B}(\W))$
satisfies the LDP with the good rate function
\begin{equation}
\hat{J}_{m}(\omega)=\inf\left\{ I(x):x\in\W,\hat{F}_{m}(x)=\omega\right\} ,\quad\forall\omega\in\W,\label{eq:36}
\end{equation}
where $\inf\emptyset=\infty$. By (\ref{32}), for each $A\in\mathscr{B}(\W)$,
$\hat{F}_{m}^{-1}(A)\in\mathscr{B}(\W)$, and
\begin{align*}
\Q_{\varepsilon}\circ\hat{F}_{m}^{-1}(A) & =\Q_{\varepsilon}\left(\hat{F}_{m}^{-1}(A)\right)\\
 & =\P\left\{ \omega\in\W:\varepsilon B(\omega)\in\hat{F}_{m}^{-1}(A)\right\} \\
 & =\P\left\{ \omega\in\W:\hat{F}_{m}\left(\varepsilon B(\omega)\right)\in A\right\} .
\end{align*}
Since $\hat{F}_{m}\left(\varepsilon B\right)=X^{\varepsilon,(m)}$
$\P$-a.s. on $\W$, we obtain that 
\[
\Q_{\varepsilon}\circ\hat{F}_{m}^{-1}(A)=\P\left\{ \omega\in\W:X^{\varepsilon,(m)}(\omega)\in A\right\} .
\]
Therefore, by the uniqueness of rate functions (see Lemma 4.1.4, Section
4.1.1 in \cite{Dembo2009}), $J_{m}$ coincides with $\hat{J}_{m}$. 

As shown in Theorem \ref{thm:Exp good X}, $\left\{ X^{\varepsilon,(m)}:\varepsilon>0\right\} $
are exponentially good approximations of $\left\{ X^{\varepsilon}:\varepsilon>0\right\} $,
so it suffices to verify that the function $I$ defined above coincides
with the function $J$ given in (\ref{eq:34}) and satisfies all conditions
in Proposition \ref{prop:Exp-good-approx}. Let us first check if
$I$ satisfies all conditions. We observe that $I$ given in (\ref{eq:33})
is a good rate function by definition. 

For any closed $C\subset\W$, denote $\eta_{m}=\inf_{\omega\in C}\hat{J}_{m}(\omega)$,
where $\hat{J}_{m}=J_{m}$ is defined as in (\ref{eq:36}). By definition,
$\eta_{m}=\inf_{\omega\in\hat{F}_{m}^{-1}(C)}I(\omega)$. Suppose
that $\liminf_{m\to\infty}\eta_{m}=\eta<\infty$, then as $I$ is
a good rate function and lower semi-continuous functions attain their
minimums on compact sets, we conclude that $I$ attains its minimum
on the closed subset $\hat{F}_{m}^{-1}(C)\in\W$. Therefore, for each
$m$ , there exists some $\omega_{m}\in\W$ such that $\omega_{m}\in\hat{F}_{m}^{-1}(C)$
and $\eta_{m}=\inf_{\omega\in\hat{F}_{m}^{-1}(C)}I(\omega)=I(\omega_{m})$. 

We notice that for all $\omega\in\W$, $\hat{F}_{m}(\omega)\to\omega$
in $\left(W,\lVert\cdot\rVert\right)$ as $m\to\infty$. Since $\hat{F}_{m}\left(\omega_{m}\right)\in C$
for all $m$, for each $\delta>0$, $\omega_{m}\in C_{\delta}$ for
large $m$, where $C_{\delta}=\left\{ \omega:\lVert\omega-C\rVert\leq\delta\right\} $.
It follows that 
\[
\inf_{\omega\in C_{\delta}}I(\omega)\leq I(\omega_{m})=\eta_{m}=\inf_{\omega\in C}\hat{J}_{m}(\omega)
\]
for $m$ sufficiently large, and hence by taking limit infimum on
both sides, we deduce that 
\[
\inf_{\omega\in C_{\delta}}I(\omega)\leq\liminf_{m\to\infty}\inf_{\omega\in C}\hat{J}_{m}(\omega).
\]
According to Lemma 4.1.6 (a), Section 4.1.1 in \cite{Dembo2009},
\begin{equation}
\inf_{\omega\in C}I(\omega)\leq\liminf_{m\to\infty}\inf_{\omega\in C}\hat{J}_{m}(\omega)\label{eq:38}
\end{equation}
when letting $\delta\to0$, and hence the condition (\ref{eq:35})
is fulfilled. The case when $\liminf_{m\to\infty}\eta_{m}=\infty$
is trivial, so we have verified all conditions in Proposition \ref{prop:Exp-good-approx}.

Next, we prove that $I$ coincides with the function $J$ defined
as in (\ref{eq:34}) by
\[
J(\omega)=\sup_{\lambda>0}\liminf_{m\to\infty}\inf_{x\in B(\omega,\lambda)}\hat{J}_{m}(x).
\]
For any $\omega\in\W$, set $C=\overline{B(\omega,\lambda)}$ in (\ref{eq:38}).
It holds that
\[
\inf_{x\in\overline{B(\omega,\lambda)}}I(x)\leq\liminf_{m\to\infty}\inf_{x\in\overline{B(\omega,\lambda)}}\hat{J}_{m}(x)\leq\liminf_{m\to\infty}\inf_{x\in B(\omega,\lambda)}\hat{J}_{m}(x)\leq J(\omega).
\]
By letting $\lambda\to0$ and applying Lemma 4.1.6 (a), Section 4.1.1
in \cite{Dembo2009}, we conclude that $I(\omega)\leq J(\omega)$
for all $\omega$. For the reverse part, denote $\hat{\omega}_{m}=\hat{F}_{m}(\omega)$
for $\omega\in\W$, and as shown above, $\hat{\omega}_{m}\to\omega$
as $m\to\infty$. Therefore, for any $\lambda>0$, there exists some
$M>0$ such that for all $m\geq M$,
\[
\inf_{x\in B(\omega,\lambda)}\hat{J}_{m}(x)\leq\hat{J}_{m}(\hat{\omega}_{m}).
\]
By the definition of $\hat{J}_{m}$, $\hat{J}_{m}(\hat{\omega}_{m})\leq I(\omega)$.
By taking limit infimum over $m$ first, then supremum over $\lambda$,
we obtain that 
\[
J(\omega)=\sup_{\lambda>0}\liminf_{m\to\infty}\inf_{x\in B(\omega,\lambda)}\hat{J}_{m}(x)\leq I(\omega),
\]
and hence $I=J$.
\end{proof}

\end{document}